\documentclass[12pt]{amsart}
\usepackage{amssymb}
\usepackage{amsmath}
\usepackage{mathtools}
\usepackage{enumitem}
\usepackage{mathrsfs}
\usepackage[all]{xy}
\usepackage{mathtools}
\usepackage{hyperref}
\usepackage{url}
\usepackage{bbm}
\usepackage{longtable}
\usepackage{verbatim}
\usepackage{enumitem}

\usepackage{lipsum}

\usepackage{lmodern}
\usepackage[OT2, T1]{fontenc}

\DeclareSymbolFont{cyrletters}{OT2}{wncyr}{m}{n}

\newtheorem{theorem}{Theorem}[section]

\newtheorem{lemma}[theorem]{Lemma}

\newtheorem{corollary}[theorem]{Corollary}

\newtheorem{proposition}[theorem]{Proposition}

\theoremstyle{remark}
\newtheorem{remark}[theorem]{Remark}

\theoremstyle{definition}
\newtheorem{definition}[theorem]{Definition}

\newtheorem*{ack}{Acknowledgements}

\usepackage[all]{xy}
\usepackage{mathtools}

\usepackage[pdftex]{graphicx}

\xyoption{web}

\newcommand{\bbA}{{\mathbb A}}

\newcommand{\bbF}{{\mathbb F}}
\newcommand{\bbG}{{\mathbb G}}

\newcommand{\bbQ}{{\mathbb Q}}
\newcommand{\bbR}{{\mathbb R}}

\newcommand{\bbZ}{{\mathbb Z}}

\newcommand{\frf}{{\mathfrak f}}

\newcommand{\cA}{{\mathcal A}}
\newcommand{\cB}{{\mathcal B}}
\newcommand{\cC}{{\mathcal C}}

\newcommand{\cE}{{\mathcal E}}

\newcommand{\cG}{{\mathcal G}}

\newcommand{\cO}{{\mathcal O}}

\newcommand{\cT}{{\mathcal T}}

\newcommand{\et}{{\textnormal{\'et}}}

\usepackage[pdftex]{graphicx}

\DeclareMathOperator{\Gal}{Gal}
\DeclareMathOperator{\Br}{Br}

\DeclareMathOperator{\Aut}{Aut}
\DeclareMathOperator{\Hom}{Hom}
\DeclareMathOperator{\End}{End}

\DeclareMathOperator{\Frob}{Frob} 

\DeclareMathOperator{\Kum}{Kum}
\DeclareMathOperator{\ord}{ord}

\DeclareMathOperator{\Cor}{Cor}
\DeclareMathOperator{\Res}{Res}
\DeclareMathOperator{\ev}{ev}
\DeclareMathOperator{\inv}{inv}

\DeclareMathOperator{\Pic}{Pic}
\DeclareMathOperator{\NS}{NS}
\DeclareMathOperator{\disc}{disc}
\DeclareMathOperator{\GL}{GL}
\DeclareMathOperator{\PGL}{PGL}

\usepackage[usenames,dvipsnames]{color}

\author{\sc Mohamed Alaa Tawfik}
\address{Mohamed Alaa Tawfik\\
Department of Mathematics\\ 
King's College London\\
Strand\\ 
London\\
WC2R 2LS\\
UK.}
   \email{mohamed.tawfik@kcl.ac.uk}
\urladdr{https://sites.google.com/view/mohamed-a-tawfik/}

\author{\sc Rachel Newton}
\address{Rachel Newton\\
Department of Mathematics\\ 
King's College London\\
Strand\\ 
London\\
WC2R 2LS\\
UK.}
   \email{rachel.newton@kcl.ac.uk}
\urladdr{https://sites.google.com/view/rachelnewton/}

\title[]{Transcendental Brauer--Manin obstructions on singular K3 surfaces}

\begin{document}



\begin{abstract}
Let $E$ and $E'$ be elliptic curves over $\bbQ$ with complex multiplication by the ring of integers of an imaginary quadratic field $K$ and let $Y=\Kum(E\times E')$ be the minimal desingularisation of the quotient of $E\times E'$ by the action of $-1$. 
We study the Brauer groups of such surfaces $Y$ and use them to furnish new examples of transcendental Brauer--Manin obstructions to weak approximation.
\end{abstract}

\maketitle

\section{Introduction}
Let $k$ be a number field and let $\bbA_k$ denote the ad\`{e}les of $k$.
Let $X/k$ be a smooth, projective, geometrically irreducible algebraic variety, let $\bar{X}$ denote its base change to an algebraic closure of $k$, and let $\Br(X)=\mathrm{H}^2_{\et}(X,\bbG_m)$ denote the Brauer group of $X$. For $v$ a place of $k$ and $\cA\in \Br(X)$, functoriality yields an evaluation map 
\begin{align*}
\ev_{\cA,v}:X(k_v)&\to \Br(k_v)\\
x_v&\mapsto \cA(x_v).
\end{align*}
 The Hasse invariant $\inv_v:\Br(k_v)\to\bbQ/\bbZ$ is an isomorphism for $v$ finite, and has image $\frac{1}{2}\bbZ/\bbZ$ for $v$ real and zero for $v$ complex. In~\cite{Manin}, Manin defined what became known as the \emph{Brauer--Manin pairing}
\begin{align}\label{eq:BM}
X(\bbA_k)\times \Br (X)&\to \bbQ/\bbZ\\
((x_v)_v \ ,\  \cA)&\mapsto \sum_{v}\inv_v(\cA(x_v))\nonumber
\end{align}
where the sum is over all places $v$ of $k$. For $B\subset \Br(X)$, the subset of $X(\bbA_k)$ consisting of all elements that are orthogonal to $B$ under the pairing~\eqref{eq:BM} is denoted $X(\bbA_k)^{B}$. The \emph{Brauer--Manin set} is 
$X(\bbA_k)^{\Br(X)}$. Global class field theory (the Albert--Brauer--Hasse--Noether Theorem), and continuity of evaluation maps, shows that $X(\bbA_k)^{\Br(X)}$ contains the closure of $X(k)$ in $X(\bbA_k)=\prod_v X(k_v)$ with respect to the product of the $v$-adic topologies. This shows that, in some cases, the emptiness of $X(k)$ despite $X$ having points in all completions can be explained by the emptiness of $X(\bbA_k)^{\Br(X)}$. This is known as a Brauer--Manin obstruction to the Hasse principle. In cases where $X(k)$ is non-empty, one would like to understand more about the rational points on $X$: for example, does weak approximation hold, i.e.~is $X(k)$ dense in $X(\bbA_k)$? If $X(\bbA_k)^B$ is not equal to $X(\bbA_k)$ for some $B\subset \Br(X)$, we say that $B$ obstructs weak approximation on $X$.

Manin's work 
 initiated a great deal of activity, 
see~\cite{Wittenbergsurvey} for a recent summary.
Initially, most research focused on the algebraic part of the Brauer group, which by definition is $\Br_1(X)=\ker(\Br(X)\to\Br(\bar{X}))$, 
and the more mysterious transcendental part $\Br(X)/\Br_1(X)$ was rarely computed. In~\cite{ISZ, I-S}, the authors computed the odd order torsion in the transcendental Brauer groups of diagonal quartic surfaces by relating these surfaces to Kummer surfaces of products of elliptic curves over $\bbQ$ with complex multiplication by $\bbZ[i]$. In~\cite{I-S, IScorrigendum}, Ieronymou and Skorobogatov went on to study the evaluation maps for these elements of odd order and thus gave new examples of Brauer--Manin obstructions to weak approximation coming from transcendental Brauer group elements. 

In this paper, we replace $\bbZ[i]$ by the ring of integers $\cO_K$ of an imaginary quadratic field $K$ and study Brauer groups and Brauer--Manin obstructions to weak approximation for Kummer surfaces of products of elliptic curves $E, E'$ over $\bbQ$ with complex multiplication by $\cO_K$. Note that the assumption that $\cO_K$ is the endomorphism ring of an elliptic curve over $\bbQ$ implies that $K$ is an imaginary quadratic field of class number one (see e.g.~\cite[Theorem~II.4.1]{SilvermanAdv}), but this is the only restriction on $K$. Moreover, our assumptions also imply that the elliptic curves $E$ and $E'$ are geometrically isomorphic, see e.g.~\cite[Proposition~II.2.1]{SilvermanAdv}.

\begin{theorem}\label{thm:mainoptions}
Let $K$ be an imaginary quadratic field and 
let $Y=\Kum(E\times E')$ for elliptic curves $E,E'$ over $\bbQ$ with $\End \bar{E}=\End\bar{E}'=\cO_K$. 
Suppose that $\Br(Y)/\Br_1(Y)$ contains an element of order $n>1$. Then $K\in \{\bbQ(\zeta_3), \bbQ(i), \bbQ(\sqrt{-7}), \bbQ(\sqrt{-2}), \bbQ(\sqrt{-11})\}$ and $n\leq 10$.
\end{theorem}

\begin{remark}
Similar results can be obtained in the more general setting where the elliptic curves can have CM by non-maximal orders in $\cO_K$, see Remark~\ref{rem:non-max} below.
\end{remark}

In Theorem~\ref{thm:mainoptions} and throughout the paper, we write $\Kum(E\times E')$ to mean the minimal desingularisation of the quotient of $E\times E'$ by the action of $-1$, which sends $(P,Q)$ to $(-P,-Q)$.
Such Kummer surfaces are examples of so-called singular K3 surfaces, which are defined to be K3 surfaces of maximal Picard rank. 

The cases of Theorem~\ref{thm:mainoptions} where $K$ is $\bbQ(i)$ or $\bbQ(\zeta_3)$ follow from work of Valloni on Brauer groups of principal K3 surfaces with CM in~\cite{Valloni}.
In the case where $K=\bbQ(i)$, the odd order torsion in the Brauer group was computed by Ieronymou--Skorobogatov--Zarhin~\cite{ISZ} and Ieronymou--Skorobogatov~\cite{I-S} in their study of Brauer groups of diagonal quartic surfaces. In particular,~\cite[Theorem~1.1]{I-S} also applies to the Kummer surface $Y=\Kum(E^{m_1}\times E^{m_2})$ where $E^m$ has affine equation $y^2=x^3-mx$ for $m\in\bbQ^\times$. It shows that 
\begin{equation}\label{ISgp}
(\Br(Y)/\Br(\bbQ))_{\textrm{odd}}=(\Br(Y)/\Br_1(Y))_{\textrm{odd}}\cong\begin{cases}\bbZ/3\bbZ & \textrm{ if } -3m_1m_2\in\langle-4\rangle\bbQ^{\times 4},\\
\bbZ/5\bbZ & \textrm{ if } 5^3m_1m_2\in\langle-4\rangle\bbQ^{\times 4},\\
0 & \textrm{ otherwise}.
\end{cases}
\end{equation}

Our next result handles all cases where $\cO_K^\times=\{\pm 1\}$.

\begin{theorem}\label{thm:mainquad}
Let $K$ be an imaginary quadratic field with $\cO_K^\times=\{\pm 1\}$ and 
let $Y=\Kum(E\times E')$ for elliptic curves $E,E'$ over $\bbQ$ with $\End \bar{E}=\End\bar{E}'=\cO_K$. Suppose that $\Br(Y)\setminus\Br(\bbQ)$ contains an element of odd order. 
Then 
\begin{enumerate}
\item \label{Kquadintro} $K\in\{\bbQ(\sqrt{-2}),\bbQ(\sqrt{-11})\}$;
\item \label{trBrquadintro} $\Br(Y)/\Br_1(Y)=\Br(Y)_3/\Br_1(Y)_3\cong\bbZ/3\bbZ$;
\item  \label{algBrquadintro} $\Br_1(Y)/\Br(\bbQ)\cong\begin{cases}(\bbZ/2\bbZ)^2&\textrm{ if }K=\bbQ(\sqrt{-2}),\\
\bbZ/2\bbZ &\textrm{ if }K=\bbQ(\sqrt{-11});
\end{cases}$
\item  \label{Yquadintro} $Y$ is the minimal desingularisation of the projective surface with affine equation $u^2=af(x)f(t)$
where \begin{align*}
&f(x)=x^3 + 4x^2 + 2x \textrm{ and }a\in\{-3,6\} &\textrm{ if }K=\bbQ(\sqrt{-2}),\\ 
&f(x)=x^3-2^5\cdot3^3\cdot11x+2^4\cdot3^3\cdot7\cdot 11^2 \textrm{ and } a\in\{-3,33\} &\textrm{ if }K=\bbQ(\sqrt{-11}).
\end{align*}
 \end{enumerate}
\end{theorem}

The remaining, and most interesting, case is $K=\bbQ(\zeta_3)$. Any elliptic curve $E$ over $\bbQ$ with $\End\bar{E}=\bbZ[\zeta_3]$ has an affine equation of the form $E^a:  y^2=x^3+a$ for some $a\in\bbQ^\times$. 
Let $c,d\in\bbQ^\times $ and let $Y=\Kum(E^c\times E^d)$. Then $\Br(Y)$ can contain transcendental elements of odd order $n$ for $n\leq 9$. Cases involving elements of order divisible by $3$ require a more delicate analysis, essentially because $3$ ramifies in the CM field $\bbQ(\zeta_3)$, and will be explored in future work. For elements of order $5$ or $7$, we have the following:

\begin{theorem}\label{thm:intro57}
For $a\in\bbQ^\times$, let $E^a$ be the elliptic curve over $\bbQ$ with affine equation $y^2=x^3+a$. Let $c,d\in\bbQ^\times $ and let $Y=\Kum(E^c\times E^d)$. Let $\ell\in\{5,7\}$ and let $\varepsilon(\ell)=(-1)^{(\ell-1)/2}$. Then
\[(\Br(Y)/\Br_1(Y))_{\ell^\infty}\cong \begin{cases}\bbZ/\ell\bbZ &\textrm{ if } \varepsilon(\ell)\cdot 2^4\cdot \ell^{\varepsilon(\ell)} \cdot cd\in \langle-3^3\rangle\bbQ^{\times 6},\\
0 &\textrm{ otherwise.}
\end{cases}\]
Furthermore, if $\varepsilon(\ell)\cdot 2^4\cdot \ell^{\varepsilon(\ell)} \cdot cd\in \langle-3^3\rangle\bbQ^{\times 6}$ then \[\Br(Y)/\Br_1(Y)=\Br(Y)_\ell/\Br_1(Y)_\ell.\]
\end{theorem}

The reason for the focus on odd order torsion in $\Br(Y)/\Br_1(Y)$ is a result of Skorobogatov and Zarhin (Theorem~\ref{thm:SZBrInject} below), which shows that odd order torsion in $\Br(E\times E')/\Br_1(E\times E')$ descends to the transcendental Brauer group of $\Kum(E\times E')$. This means that one can transfer many calculations concerning transcendental Brauer classes to the realm of abelian varieties. In particular, the method of Skorobogatov and Zarhin described in Section~\ref{sec:prelim_eval} enables one to compute Brauer--Manin pairings for transcendental Brauer group elements of odd order without the need to find explicit Azumaya algebras representing them.

However, it is important that we consider Kummer surfaces and not just abelian surfaces, for the following reason. For torsors under abelian varieties over number fields, any Brauer--Manin obstruction to the Hasse principle or weak approximation can already be explained by an algebraic element in the Brauer group, see~\cite{Creutz, Manin}. In contrast, for K3 surfaces, it can happen that the algebraic part of the Brauer group consists only of constant elements (and so does not contribute to any Brauer--Manin obstruction), while there is an obstruction coming from a transcendental element in the Brauer group. Examples of this phenomenon were given in~\cite{Ieronymou, transBrprod, MSTVA}. Our next two results yield a new source of examples.

\begin{theorem}\label{thm:mainWA}
Let $K$ be an imaginary quadratic field and 
let $Y=\Kum(E\times E')$ for elliptic curves $E,E'$ over $\bbQ$ with $\End \bar{E}=\End\bar{E}'=\cO_K$. Let $\ell$ be an odd prime and if $K=\bbQ(\zeta_3)$ assume that $\ell>3$. 
Suppose that $\cA\in\Br(Y)_{\ell}\setminus \Br(\bbQ)$. Then the evaluation map $\ev_{\cA,\ell}: Y(\bbQ_\ell)\to\Br(\bbQ_\ell)_\ell$ is surjective and hence
\[Y(\bbA_\bbQ)^\cA\neq Y(\bbA_\bbQ),\] 
i.e.\ $\cA$ obstructs weak approximation on $Y$.
\end{theorem}

\begin{theorem}\label{thm:introtransob}
Reinstate the notation and assumptions of Theorem~\ref{thm:intro57}. Suppose that $\varepsilon(\ell)\cdot 2^4\cdot \ell^{\varepsilon(\ell)} \cdot cd\in \langle-3^3\rangle\bbQ^{\times 6}$ and $c\notin 2\cdot \ell^{\varepsilon(\ell)}\cdot \bbQ^{\times 3}$. 
Then 
\[\Br_1(Y)=\Br(\bbQ)\] 
and hence the failure of weak approximation in Theorem~\ref{thm:mainWA} cannot be explained by any algebraic element in the Brauer group of $Y$.
\end{theorem}

\begin{remark}
\begin{enumerate}

\item In Theorem~\ref{thm:const} we prove a complement to Theorem~\ref{thm:mainWA}, showing in many cases that the evaluation maps $\ev_{\cA,v}$ for places $v\neq \ell$ are constant, cf.~\cite[Theorem~1.2(i)]{I-S}.

\item Creutz and Viray showed in~\cite[Theorem~1.7]{CV} that if there is a Brauer--Manin obstruction to the Hasse principle on a Kummer variety, then there is an obstruction coming from an element in the $2$-primary part of the Brauer group. 
It was already known (see e.g.~\cite{transBrprod}) that the analogous statement for Brauer--Manin obstructions to weak approximation does not hold. Theorem~\ref{thm:introtransob} gives further illustration of this fact, generalising the example given in~\cite[Theorem~1.3]{transBrprod}. We note that the statement of~\cite[Theorem~1.3]{transBrprod} needs correcting: for $\cA\in\Br(X)_3\setminus \Br(\bbQ)$, the evaluation map $\ev_{\cA,v}:X(\bbQ_v)\to \Br(\bbQ_v)_3$ is constant, but not necessarily zero, for all $v\neq 3$. However, any such $\cA$ does give an obstruction to weak approximation and there is a choice of $\cA$ for which the theorem holds as stated -- one takes $\cA=\cB$, where the notation is as in the proof of Theorem~\ref{thm:const}. This choice of $\cA$ should be in force throughout~\cite[Section~5]{transBrprod}, yielding corrections to the statements of \cite[Proposition~5.1, Theorem~5.2, Theorem~5.3]{transBrprod}, see~\cite{trBrcorr}.

\item In the same way that Ieronymou, Skorobogatov and Zarhin used Kummer surfaces of products of elliptic curves with CM by $\bbZ[i]$ to study Brauer groups and Brauer--Manin obstructions on diagonal quartic surfaces in~\cite{ISZ,I-S,IScorrigendum}, the results of this paper could be applied to the study of other families of quartic surfaces, e.g.~those of the form $ax^4+cxy^3=bz^4+dzw^3$ with $abcd\in\bbQ^\times$, in other words the family of surfaces geometrically isomorphic to Schur's quartic surface. Note that these surfaces contain lines given by $x=z=0$ so any Brauer--Manin obstructions arising would be obstructions to weak approximation.
\end{enumerate}
\end{remark}

\subsection*{Outline of the paper}
We begin by gathering some preliminary results on transcendental elements of Brauer groups and their evaluations at local points in Section~\ref{sec:prelim}. In Section~\ref{sec:gpquad}, for $K$ an imaginary quadratic field with $\cO_K^\times=\{\pm 1\}$, we compute transcendental Brauer groups of products of elliptic curves over $\bbQ$ with CM by $\cO_K$ and prove Theorem~\ref{thm:mainoptions}. In Section~\ref{sec:zetagp} we perform the same calculation in the case where $K=\bbQ(\zeta_3)$ and prove Theorem~\ref{thm:intro57}. In Section~\ref{sec:Br1} we compute the algebraic part of the Brauer group for each of the Kummer surfaces under consideration and prove Theorem~\ref{thm:introtransob}. Combining the results for the transcendental and algebraic parts of the Brauer group allows us to prove Theorem~\ref{thm:mainquad}. In Section~\ref{sec:pair_const} we consider the evaluation of a Brauer group element of prime order $\ell$ at $p$-adic points with $p\neq \ell$ and show in many cases that these evaluation maps are constant, providing a complement to Theorem~\ref{thm:mainWA}. Section~\ref{sec:pair_nonconst} is devoted to the proof of Theorem~\ref{thm:mainWA}.

\subsection{Notation}
If $A$ is an abelian group and $n$ a positive integer, then $A_n$ and $A/n$
denote the kernel and cokernel, respectively, of multiplication by $n$ on $A$. If $\ell$ is
prime, then $A_{\ell^\infty}$ denotes the $\ell$-power torsion subgroup of $A$.

If $k$ is a field of characteristic zero, then $\bar{k}$ denotes an algebraic closure of $k$ and $\Gamma_k$ denotes the absolute Galois group $\Gal(\bar{k}/k)$. If $X$ is an algebraic variety over $k$ and $l/k$ is a field extension then $X_l$ denotes the base change $X\times_k l$. The base change $X\times_k \bar{k}$ is denoted by $\bar{X}$.

The Brauer group of $X/k$ is denoted by $\Br (X)$ and its algebraic part  
$\ker(\Br( X) \to \Br( \bar{X}))$ is denoted by $\Br_1(X)$. The quotient $\Br(X)/\Br_1(X)$ is called the transcendental part of $\Br(X)$, or the transcendental Brauer group of $X$.

For an elliptic curve $E/k$ we denote by $\End\bar{E}$ the full ring of endomorphisms defined over $\bar{k}$. 

\begin{ack}
We are grateful to David Kurniadi Angdinata, Tudor Ciurca, Netan Dogra, Jack Shotton, Alexei Skorobogatov and Alex Torzewski for useful discussions. We thank the anonymous reviewers for their thorough reading of an earlier draft of this paper and for many insightful comments and suggestions that have significantly improved it, including pointing out an error in a previous version of Theorem~\ref{thm:Br1}. Mohamed Alaa Tawfik was supported by a University of Reading International Research Studentship and a King's College London Faculty of Natural, Mathematical and Engineering Sciences Studentship. Rachel Newton was supported by EPSRC grant EP/S004696/1 and EP/S004696/2,
and UKRI Future Leaders Fellowship MR/T041609/1 and MR/T041609/2. \end{ack}

\section{Preliminaries}\label{sec:prelim}

\subsection{Transcendental Brauer groups}

The following result of Skorobogatov and Zarhin allows us to move between the transcendental Brauer group of an abelian surface and that of the associated Kummer surface.

\begin{theorem}[{\cite[Theorem~2.4]{SZtorsion}}]\label{thm:SZBrInject} 
Let $A$ be an abelian surface, let $Y=\Kum(A)$ and let $n\in\bbZ_{>0}$. There is a natural embedding
\begin{equation}\label{eq:Br_inj}
\Br (Y)_n/\Br_1 (Y)_n\hookrightarrow \Br (A)_n/\Br_1 (A)_n
\end{equation}
which is an isomorphism if $n$ is odd. The subgroups of elements of odd order of the transcendental Brauer groups $\Br (Y)/\Br_1 (Y)$ and $\Br (A)/\Br_1 (A)$ are isomorphic.
\end{theorem}

To calculate the transcendental part of the Brauer group for products of elliptic curves, we use another result of Skorobogatov and Zarhin.

\begin{proposition}[{\cite[Proposition~3.3]{SZtorsion}}]\label{prop:SZhom}
Let $E$ and $E'$ be elliptic curves over a field $k$ of characteristic zero. 
For $n\in\bbZ_{>0}$, we have a canonical isomorphism of abelian groups
\[\Br (E\times E')_n/\Br_1 (E\times E')_n=\Hom_{\Gamma_k}(E_n,E'_n)/(\Hom(\bar{E},\bar{E}')/n)^{\Gamma_k}.\]
\end{proposition}

When the elliptic curves have complex multiplication, we can say more. 

\begin{definition}
Let $k$ be a field of characteristic zero 
and let $E,E'$ be elliptic curves over $k$ with $\bar{E}\cong \bar{E}'$ and
$\End\bar{E}=\End\bar{E}'=\cO_K$ for an imaginary quadratic field $K$.
For $n\in\bbZ_{>0}$, we define $\Gamma_k$-submodules of $\Hom(E_n,E'_n)$, as follows:
\begin{align*}
\Hom(E_n,E'_n)^+&=\{\varphi\in\Hom(E_n,E'_n)\mid\varphi \circ\beta=\beta\circ\varphi \ \forall \beta\in \cO_K \};\\
\Hom(E_n,E'_n)^-&=\{\varphi\in\Hom(E_n,E'_n)\mid\varphi\circ \beta=\bar{\beta}\circ\varphi \ \forall \beta\in \cO_K \},
\end{align*}
where $\bar{\beta}$ denotes the complex conjugate of $\beta$.
\end{definition}

The following lemma and corollary are slight generalisations of some results from~\cite[\S3]{I-S}.

\begin{lemma}\label{lem:+-}
Let $k$ be a field of characteristic zero. 
Let $d\in\bbZ_{>0}$ be squarefree and let $K=\bbQ(\sqrt{-d})$. Let $E,E'$ be elliptic curves over $k$ with $\bar{E}\cong \bar{E}'$ and $\End\bar{E}=\End\bar{E}'=\cO_K$. For all $n\in\bbZ_{>0}$ coprime to $2d$, we have equalities of $\Gamma_k$-modules
\begin{equation}\label{eq:hom+-}
\Hom(E_n,E'_n)=\Hom(E_n,E'_n)^+\oplus \Hom(E_n,E'_n)^-
\end{equation}
and 
\begin{equation}\label{eq:hom+}
\Hom(\bar{E},\bar{E}')/n=\Hom(E_n,E'_n)^+.
\end{equation}
In the case $d=3$,~\eqref{eq:hom+-} and~\eqref{eq:hom+} hold for all $n\in\bbZ_{>0}$ coprime to $3$.
\end{lemma}

\begin{proof}
First suppose that $n$ is coprime to $2d$. Then multiplication by $\sqrt{-d}$ is invertible on $E_n$ and $E_n'$ and we have
\begin{align*}
\Hom(E_n,E'_n)^+&=\{\varphi\in\Hom(E_n,E'_n)\mid\sqrt{-d}\circ \varphi\circ\sqrt{-d}^{-1}=\varphi \};\\
\Hom(E_n,E'_n)^-&=\{\varphi\in\Hom(E_n,E'_n)\mid\sqrt{-d}\circ \varphi\circ\sqrt{-d}^{-1}=-\varphi\}.
\end{align*}
Since $n$ is odd, we can write any $\varphi\in \Hom(E_n,E'_n)$ as 
\[\varphi=\frac{1}{2}\left(\varphi+ \sqrt{-d}\circ \varphi\circ\sqrt{-d}^{-1}\right)+\frac{1}{2}\left(\varphi-\sqrt{-d}\circ \varphi\circ\sqrt{-d}^{-1}\right)\]
and thus prove~\eqref{eq:hom+-}.

The case $d=3$ and $n$ coprime to $3$ is similar except that we consider conjugation by $\zeta_3$ instead of $\sqrt{-d}$. Since $\zeta_3^3=1$, we have
\begin{align*}
\Hom(E_n,E'_n)=\Hom(E_n,E'_n)^0\oplus \Hom(E_n,E'_n)^1\oplus \Hom(E_n,E'_n)^2
\end{align*}
where $\Hom(E_n,E'_n)^j=\{\varphi\in \Hom(E_n,E'_n)\mid \zeta_3\circ\varphi\circ\zeta_3^{-1}=\zeta_3^j\circ\varphi\}$. But if $\varphi\in \Hom(E_n,E'_n)^1$ then $0=\varphi\circ(\zeta_3^{-1}-1)$, which implies that $\varphi\circ 3=0$ and hence $\varphi=0$, as $n$ is coprime to $3$. Therefore, $\Hom(E_n,E'_n)^1=0$. Now observe that $\Hom(E_n,E'_n)^0=\Hom(E_n,E'_n)^+$ and $\Hom(E_n,E'_n)^2=\Hom(E_n,E'_n)^-$ to complete the proof of~\eqref{eq:hom+-}.
 
For~\eqref{eq:hom+}, view $\End(\bar{E})$ and $\End(E_n)$ as $\cO_K^\times$-modules via the action of $\cO_K^\times=\Aut\bar{E}$ on the second factor in each case, so that $\alpha\in\cO_K^\times$ sends an endomorphism $\varphi$ to $\alpha\varphi$. Since $\bar{E}\cong\bar{E}'$, the elliptic curve $E'$ is a twist of $E$ by an element in $\mathrm{H}^1(k,\Aut\bar{E})=\mathrm{H}^1(k,\cO_K^\times)$. 
Thus, the $\Gamma_k$-modules $\Hom(\bar{E},\bar{E}')$ and $\Hom(E_n,E'_n)$ are twists of $\End (\bar{E})$ and $\End(E_n)$, respectively, by the same element of $\mathrm{H}^1(k,\cO_K^\times)$, acting on the second factor. Therefore, it is enough to prove that $\End (\bar{E})/n=\End(E_n)^+$. Since $\End (\bar{E})=\cO_K$, it is clear that $\End (\bar{E})/n\subset \End(E_n)^+$. Equality follows from the fact that both are isomorphic to $(\bbZ/n\bbZ)^2$ as abelian groups.
\end{proof}

\begin{corollary}\label{cor:Br-}
Under the assumptions of Lemma~\ref{lem:+-}, we have a canonical isomorphism of abelian groups
\[\Br (E\times E')_n/\Br_1 (E\times E')_n=\Hom_{\Gamma_k}(E_n,E'_n)^-.\]
\end{corollary}

\begin{proof}
Follows immediately from Proposition~\ref{prop:SZhom} and Lemma~\ref{lem:+-}.
\end{proof}

In order to use Proposition~\ref{prop:SZhom} to calculate the whole of the transcendental part of the Brauer group, we will need the following material from~\cite{transBrprod}.

\begin{definition}\label{def:n}
Fix a number field $L$, an imaginary quadratic field $K$ and a prime number $\ell\in\bbZ$. Define $n(\ell)$ to be the largest integer $t$ such that the ring class field $K_{\ell^t}$ corresponding to the order $\bbZ+\ell^t\cO_K$ embeds into $KL$.
\end{definition}

\begin{theorem}\label{thm:RN}
Let $L$ be a number field and let $E/L$ be an elliptic curve such that $\End \bar{E}=\cO_K$ for an imaginary quadratic field $K$. Then for all prime numbers $\ell\in\bbZ$,
\[\left(\frac{\Br(E\times E)}{\Br_1(E\times E)}\right)_{\ell^\infty}=\frac{\Br(E\times E)_{\ell^{n(\ell)}}}{\Br_1(E\times E)_{\ell^{n(\ell)}}}.\]
\end{theorem}

\begin{proof}
This is an immediate consequence of~\cite[Proposition~2.2, Theorem~2.5 and Theorem~2.9]{transBrprod}.
\end{proof}

To aid us in our applications of Theorem~\ref{thm:RN}, we will need the following well-known formula for the degree of a ring class field (see \cite[Theorem 7.24]{Cox}, for example). Let $K$ be an imaginary quadratic field with discriminant $\Delta_K$ and class number $h_K$, let $c\in\bbZ_{>0}$ and let $\cO_c=\bbZ+c\cO_K$ be the order of conductor $c$ in $\cO_K$. Then
\begin{equation}\label{eq:ringclass}
[K_{c} :K]= \frac{h_K\cdot c}{[\cO_K^\times:\cO_c^\times]}\prod_{p\mid c}\left(1-\left(\frac{\Delta_K}{p}\right)\frac{1}{p}\right).
\end{equation}
The symbol $(\frac{\cdot}{p})$ denotes the Legendre symbol for odd primes. For the prime $2$, the
Legendre symbol is replaced by the Kronecker symbol $(\frac{\cdot}{2} )$, with
\[\left(\frac{\Delta_K}{2}\right)=\begin{cases}
0 & \textrm{ if } 2\mid\Delta_K,\\
1& \textrm{ if }\Delta_K\equiv 1\pmod{8},\\
-1& \textrm{ if } \Delta_K\equiv 5   \pmod{8}.
\end{cases}\]

\subsection{Evaluation maps}\label{sec:prelim_eval}

Let $k$ be a number field. 
Let $d\in\bbZ_{>0}$ be squarefree and let $K=\bbQ(\sqrt{-d})$.
Let $E,E'$ be elliptic curves over $k$ with $\bar{E}\cong \bar{E}'$ and $\End\bar{E}=\End\bar{E}'=\cO_K$ and let $Y=\Kum(E\times E')$.
Let $n\in\bbZ_{>0}$ be coprime to $2d$ and let $\varphi\in \Hom_{\Gamma_k}(E'_n, E_n)^-$. For the reader's convenience, following~\cite[\S5.1]{I-S},~\cite[\S3]{SZtorsion}, we summarise here the construction of an element of $\Br(Y)_n$ from $\varphi$, and describe its evaluation at a $p$-adic point of $Y$ in terms of a cup-product map.

Multiplication by $n$ on $E$ turns $E$ into an $E$-torsor with structure group $E_n$. Denote this torsor by $\cT$ and let $[\cT]$ denote its class in $\mathrm{H}^1_{\et}(E, E_n)$. Similarly, let $\cT'$ denote $E'$ considered as an $E'$-torsor with structure group $E'_n$, and let $[\cT']$ denote its class in $\mathrm{H}^1_{\et}(E', E'_n)$. The homomorphism $\varphi: E'_n\to E_n$ gives rise to the $E'$-torsor $\varphi_*\cT'$ with structure group $E_n$, with class $[\varphi_*\cT'] \in \mathrm{H}^1_{\et}(E', E_n)$.

Composing the cup-product map with the Weil pairing $E_n\times E_n\to \mu_n$ yields a pairing
\begin{equation}\label{eq:torsorpair}
\mathrm{H}^1_{\et}(E\times E', E_n)\times \mathrm{H}^1_{\et}(E\times E', E_n)\to \Br(E\times E')_n.
\end{equation}
Let $p:E\times E'\to E$ and $p':E\times E'\to E'$ be the natural projection maps. The pullbacks $p^*\cT_1$ and $p'^*\varphi_*\cT'$ are $E\times E'$-torsors with structure group $E_n$; let $\cC\in \Br(E\times E')_n$ denote the pairing of their classes in $\mathrm{H}^1_{\et}(E\times E', E_n)$ via~\eqref{eq:torsorpair}. By~\cite[Lemma~3.1, Proposition~3.3]{SZtorsion}, the natural map
\[\Br(E\times E')_n\to \Br(E\times E')_n/\Br_1(E\times E')_n=\Hom_{\Gamma_k}(E'_n, E_n)^-\]
sends $\cC$ to $\varphi$. Let $\iota$ denote the involution on $\Br(E\times E')$ induced by $(P,Q)\mapsto (-P,-Q)$ on $E\times E'$. The proof of~\cite[Theorem~2.4]{SZtorsion} identifies $\Br(Y)$ with the subgroup of $\Br(E\times E')$ consisting of elements fixed by $\iota$. By the functoriality and bilinearity of the cup product, we find that $\iota(\cC)=\cC$. Let $\cB$ be the element of $\Br(Y)$ corresponding to $\cC$. 

Note that if we take quadratic twists of $E$ and $E'$ by the same element $a\in k^\times$ then there is a natural isomorphism $\Kum(E^a\times E'^a)\to Y$. Applying the construction described above to the homomorphism of $\Gamma_k$-modules $E'^a_n\to E^a_n$ coming from $\varphi$, we obtain an element of the Brauer group of $\Kum(E^a\times E'^a)$ that is identified with $\cB\in\Br(Y)$ under the isomorphism $\Kum(E^a\times E'^a)\to Y$.

Let $F$ be a field containing $k$ (e.g. $F$ could be the completion of $k$ at some place $v$ of $k$). Let $P\in E(F)\setminus E_2$, let $Q\in E'(F)\setminus E_2$ and let $[P,Q]\in Y(F)$ denote the corresponding point on the Kummer surface. Then 
\begin{equation*}
\cB([P,Q])=\cC((P,Q)).
\end{equation*}
Let $\chi_P$ denote the image of $P$ under the natural map $\chi:E(F)\to \mathrm{H}^1(F, E_n)$ and let $\chi_Q$ denote the image of $Q$ under the natural map $\chi:E'(F)\to \mathrm{H}^1(F, E'_n)$. The cup product and the Weil pairing $E_n\times E_n\to \mu_n$ give a pairing
\begin{equation}\label{eq:cup}
\cup: \mathrm{H}^1(F, E_n)\times \mathrm{H}^1(F, E_n)\to \Br(F)_n.
\end{equation}
Now the construction of $\cC$ and the functoriality of the cup product show that 
\begin{equation}\label{eq:pairingcup}
\cB([P,Q])=\cC((P,Q))=\chi_P\cup\varphi_*(\chi_Q)\in \Br(F)_n.
\end{equation}

This description of the evaluation map, due to Skorobogatov and Zarhin in~\cite{SZtorsion}, is very powerful because it enables one to evaluate transcendental elements of $\Br(Y)$ at local points (and thus compute the Brauer--Manin pairing) without the need to obtain explicit Azumaya algebras representing these elements of the Brauer group.

\begin{lemma}\label{lem:constant0}
Let $k$ be a number field and let $k_v$ be its completion at a place $v$. Let $d\in\bbZ_{>0}$ be squarefree and let $K=\bbQ(\sqrt{-d})$.
Let $E,E'$ be elliptic curves over $k$ with $\bar{E}\cong \bar{E}'$ and $\End\bar{E}=\End\bar{E}'=\cO_K$ and let $Y=\Kum(E\times E')$.
Let $n\in\bbZ_{>0}$ be coprime to $2d$ and let $\varphi\in \Hom_{\Gamma_k}(E'_n, E_n)^-$. Let $\cB\in \Br(Y)$ be constructed from $\varphi$ as described above.
If the evaluation map 
\begin{align*}
\ev_{\cB, v}: Y(k_v)&\to \Br(k_v)_n\\
y&\mapsto \cB(y)
\end{align*}
 is constant then it is zero.
\end{lemma}

\begin{proof}
Let $P\in E(k_v)\setminus E_4$ (so that $2P\notin E_2$) and let $Q\in E'(k_v)\setminus E_2$. If $\ev_{\cB}$ is constant then $\cB([2P,Q])=\cB([P,Q])$ and~\eqref{eq:pairingcup} gives
\[\cB([2P,Q])=\chi_{2P}\cup\varphi_*(\chi_Q)=2\cdot(\chi_P\cup\varphi_*(\chi_Q))=2\cB([P,Q]).\]
Hence $\cB([P,Q])=0$, which suffices to prove the lemma.
\end{proof}

\section{CM by $\cO_K$ with $\cO_K^\times=\{\pm 1\}$: transcendental Brauer groups}\label{sec:gpquad}

If $\End\bar{E}=\cO_K$ for an imaginary quadratic field $K$ with $\cO_K^\times=\{\pm 1\}$ (i.e.\ $K$ is not $\bbQ(i)$ or $\bbQ(\zeta_3)$), then 
the only twists of $E$ are quadratic twists. Theorem~\ref{thm:quad} below shows that in this case the transcendental part of the Brauer group of $E\times E'$ has exponent at most $6$, where $E'$ denotes a quadratic twist of $E$. The theorem is stated for elliptic curves over $K$, but the conclusion also holds for elliptic curves over $\bbQ$, by definition of the transcendental part of the Brauer group (cf.~\eqref{eq:L1} below).

\begin{theorem}\label{thm:quad}
Let $K$ be an imaginary quadratic field with $\cO_K^\times=\{\pm 1\}$. Let $E/K$ be an elliptic curve with~$\End\bar{E}=\cO_K$, let $E'$ be a quadratic twist of $E$ and let $T=\Br(E\times E') /\Br_1(E\times E')$. 
\begin{enumerate}[label=(\roman*)]
\item\label{-2} If $K=\bbQ(\sqrt{-2})$ then $T$ is killed by $6$.
\item If $K=\bbQ(\sqrt{-7})$ then $T$ is killed by $4$.
\item\label{-11} If $K=\bbQ(\sqrt{-11})$ then $T$ is killed by $3$.
\item In all other cases, $T=0$. 
\end{enumerate}
\end{theorem}

\begin{proof}
Let $E'$ be the quadratic twist of $E$ by some $a\in K^\times$ and let $L=K(\sqrt{a})$. Write $A=E\times E'=E\times E^a$. Observe that $A_L$ is isomorphic to $E\times E$ over $L$ and by definition we have
\begin{equation}\label{eq:L1}
\Br(A)/\Br_1(A)\hookrightarrow \Br(A_L)/\Br_1(A_L).
\end{equation}
Applying~\cite[Theorem~1.1 and Proposition~2.2]{transBrprod}, we see that for any prime number $\ell$,
\begin{equation}\label{eq:n}
(\Br(A_L)/\Br_1(A_L))_{\ell^\infty}=\Br(A_L)_{\ell^{n(\ell)}}/\Br_1(A_L)_{\ell^{n(\ell)}}
\end{equation}
where $n(\ell)$ is the largest integer $t$ such that the ring class field $K_{\ell^t}$ corresponding to the order $\bbZ+\ell^t\cO_K$ embeds into $L$. Bounds on $n(\ell)$ are easily obtained by noting that if $K_{\ell^t}$ embeds into $L$ then $[K_{\ell^t}:K]$ divides $[L:K]\leq 2$. Furthermore, we can use the formula~\eqref{eq:ringclass} to calculate $[K_{\ell^t}:K]$, noting that the theory of complex multiplication shows that $h_K=1$, since the Hilbert class field $H_K$ is equal to $K(j(E))$ and $E$ is defined over $K$. In this way, we find that $n(\ell)=0$ for all $\ell\geq 5$. For $\ell=3$, we find that $n(3)\leq 1$, and $n(3)=0$ unless $K$ is an imaginary quadratic field of class number one with $\left(\frac{\Delta_K}{3}\right)=1$, i.e.\ unless $K\in \{\bbQ(\sqrt{-2}), \bbQ(\sqrt{-11})\}$. Similarly, we find that $n(2)=0$ unless $K\in \{\bbQ(\sqrt{-7}), \bbQ(\sqrt{-2})\}$, and furthermore $n(2)\leq 1 $ if $K=\bbQ(\sqrt{-2})$ and $n(2)\leq 2$ if $K=\bbQ(\sqrt{-7})$.
\end{proof}

\begin{proof}[Proof of Theorem~\ref{thm:mainoptions}]
As noted in the introduction, the assumptions of Theorem~\ref{thm:mainoptions} imply that the CM field $K$ has class number one and the elliptic curves $E$ and $E'$ are geometrically isomorphic. If $\cO_K^\times=\{\pm 1\}$, then this means that $E'$ is a quadratic twist of $E$ and the result follows from Theorems~\ref{thm:quad} and~\ref{thm:SZBrInject}. The remaining cases, where $K\in\{\bbQ(i), \bbQ(\zeta_3)\}$, follow from~\cite[Examples~1 and~2, pp.48--51]{Valloni} and Theorem~\ref{thm:SZBrInject}. 
\end{proof}

\begin{remark}\label{rem:non-max}
If we relax the assumptions of Theorem~\ref{thm:mainoptions} to allow $E$ and $E'$ to have CM by orders $\cO_\frf$ and $\cO_{\frf'}$ in $\cO_K$ of conductors $\mathfrak{f}$ and $\mathfrak{f}'$, respectively, we can obtain similar results using the existence of isogenies of degrees $\frf$ and $\frf'$ from $E$ and $E'$, respectively, to elliptic curves with CM by $\cO_K$. Since $K_\frf=K(j(E))=K=K(j(E'))=K_{\frf'}$ (see e.g.\ \cite[Theorem~11.1]{Cox}), the formula~\eqref{eq:ringclass} shows that $\frf,\frf'\leq 3$ and $K\in \{\bbQ(\zeta_3), \bbQ(i), \bbQ(\sqrt{-7}), \bbQ(\sqrt{-2}), \bbQ(\sqrt{-11})\}$. Bounds on the order of a non-trivial class in the transcendental Brauer group then follow from~\cite[Theorem~5.13]{BJN} and Theorem~\ref{thm:SZBrInject}.
\end{remark}

Thus, in the setting of Theorem~\ref{thm:quad}, an element of odd order in $\Br(E\times E')/\Br_1(E\times E')$ has order dividing $3$. In Section~\ref{sec:zetagp}, we will see that the situation is more interesting in the case of elliptic curves with complex multiplication by $\bbZ[\zeta_3]$, where sextic twists can occur. But first we will investigate the cases in Theorem~\ref{thm:quad} where non-trivial elements of odd order can occur in the transcendental part of the Brauer group, namely when $K\in\{\bbQ(\sqrt{-2}),\bbQ(\sqrt{-11})\}$. Elements of odd order are of particular interest to us because Theorem~\ref{thm:SZBrInject} shows that they descend to the transcendental part of the Brauer group of the relevant Kummer surface, where there is a chance they may give obstructions to weak approximation that cannot be explained by any algebraic element in the Brauer group.

\begin{lemma}\label{lem:3}
Let $K\in \{\bbQ(\sqrt{-2}),\bbQ(\sqrt{-11})\}$, let $F\in \{\bbQ, K\}$, let $E/F$ be an elliptic curve such that $\End\bar{E}=\cO_K$ and let $a\in F^\times$. Then
\[\left(\frac{\Br(E\times E^a)}{\Br_1(E\times E^{a})}\right)_{3^\infty}=\frac{\Br(E\times E^a)_3}{\Br_1(E\times E^a)_3}.\]
\end{lemma}

\begin{proof}
Let $A=E\times E^a$.
Since $\Br(A)$ is torsion (see~\cite[Lemma~3.5.3]{CTSbook}), 
\[(\Br(A)/\Br_1(A))_{3^\infty}=\Br(A)_{3^\infty}/\Br_1(A)_{3^\infty}.\]
Let $\cA\in \Br (A)_{3^\infty}$ and let $L=K(\sqrt{a})$. Then~\eqref{eq:L1}, \eqref{eq:n} and the fact that $n(3)\leq 1$ (see the proof of Theorem~\ref{thm:quad}) give
\begin{equation}\label{eq:n3}
(\Br(A)/\Br_1(A))_{3^\infty}\hookrightarrow \Br(A_L)_{3}/\Br_1(A_L)_{3}.
\end{equation}
Hence, $\Res_{L/F}\cA=\cB+\cC$ where $\cB\in \Br(A_L)_{3}$ and $\cC\in \Br_1(A_L)_{3^\infty}$. Applying corestriction yields 
\[[L:F]\cdot\cA=\Cor_{L/F}\cB+\Cor_{L/F}\cC,\] 
where $\Cor_{L/F}\cB\in \Br(A)_{3}$ and $\Cor_{L/F}\cC\in \Br_1(A)_{3^\infty}$, by~\cite[Lemme~1.4]{CTS}. Since $[L:F]$ is coprime to $3$, we can invert $[L:F]$ modulo the order of $\cA$ to see that the class of $\cA$ in $\Br(A)_{3^\infty}/\Br_1(A)_{3^\infty}$ lies in $\Br(A)_{3}/\Br_1(A)_{3}$.
\end{proof}

\begin{proposition}\label{prop:alpha}
Let $K\in \{\bbQ(\sqrt{-2}),\bbQ(\sqrt{-11})\}$, let $E/K$ be an elliptic curve with CM by $\cO_K$, let $a\in K^\times$ and let $E^a$ denote the quadratic twist of $E$ by $a$. Suppose that $\Br(E\times E^a)/\Br_1(E\times E^a)$ contains an element of order $3$. Then $K(\sqrt{a})=K(\sqrt{-3})$ and hence $a\in -3\cdot K^{\times 2}$.
\end{proposition}

\begin{proof}
The proof of Theorem~\ref{thm:quad} shows that if the $3$-primary part of $\Br (E\times E^a)/\Br_1(E\times E^a)$ is non-trivial (so necessarily $n(3)>0$) then $K(\sqrt{a})=K_3$. It is easily checked that $K_3=K(\sqrt{-3})$.
\end{proof}

\begin{theorem}\label{thm:numerator3}
Let $K\in \{\bbQ(\sqrt{-2}),\bbQ(\sqrt{-11})\}$, let $F\in \{\bbQ, K\}$ and let $E/F$ be an elliptic curve such that $\End\bar{E}=\cO_K$. Furthermore, let $a\in F^\times \cap -3\cdot K^{\times 2}$. Then 
\begin{align*}
\Br(E\times E^a)/\Br_1(E\times E^a)&= \frac{\Br(E\times E^a)_3}{\Br_1(E\times E^{a})_3} \\
&=\Hom_{\Gamma_F}(E_3,E^{a}_3)^-\\
&\cong \begin{cases}
(\bbZ/3\bbZ)^2 &\textrm { if } F=K;\\
\bbZ/3\bbZ &\textrm { if } F=\bbQ.
\end{cases}
\end{align*}
\end{theorem}

\begin{remark}
In Theorem~\ref{thm:numerator3}, if $F=K$ then one may assume $a=-3$ since multiplying $a$ by an element of $F^{\times 2}$ does not change $E^a$. Likewise, if $F=\bbQ$ then one may assume $a\in \{-3,6\}$ if $K=\bbQ(\sqrt{-2})$, and $a\in \{-3,33\}$ if $K=\bbQ(\sqrt{-11})$. Proposition~\ref{prop:alpha} shows that these are the only quadratic twists for which $\Br(E\times E^a)/\Br_1(E\times E^a)$ contains non-trivial elements of odd order. 
\end{remark}

\begin{proof}[Proof of Theorem~\ref{thm:numerator3}]
We begin by computing $\Hom_{\Gamma_F}(E_3,E^{a}_3)^-$. Since $K$ has class number one, the theory of complex multiplication shows that any two elliptic curves over $F$ with CM by $\cO_K$ are geometrically isomorphic and hence quadratic twists of each other, in this case. Therefore, we can select a chosen elliptic curve $\cE/F$ with CM by $\cO_K$ and write $E=\cE^\delta$ for some $\delta\in F^\times$. Thus, $E^a=\cE^{\delta a}$ and $\Hom(E_3,E^{a}_3)=\Hom(\cE_3^\delta,\cE^{a\delta}_3)$. Now $\Gamma_F$ acts on $\Hom(\cE_3^\delta,\cE^{a\delta}_3)$ by conjugation and the two quadratic twists by $\delta$ cancel each other out so that $\Hom(\cE_3^\delta,\cE^{a\delta}_3)=\Hom(\cE_3,\cE^{a}_3)$ as a $\Gamma_F$-module. 
Furthermore, the $\Gamma_F$-module $\Hom(\cE_3,\cE^{a}_3)^-$ is the quadratic twist of $\End( \cE_3)^-$ by the quadratic character corresponding to $F(\sqrt{a})/F$. In other words, 
we identify $\Hom(\cE_3,\cE^{a}_3)^-$ with the group $\End( \cE_3)^-$ equipped with an action of $\Gamma_F$ such that $\sigma\in\Gamma_F$ sends $\varphi\in \End( \cE_3)^-$ to $\frac{\sigma(\sqrt{a})}{\sqrt{a}}\sigma\varphi\sigma^{-1}$.

\paragraph{$\bullet\  K=\bbQ(\sqrt{-2})$:} We take $\cE$ to be the elliptic curve with affine equation $y^2 = x^3 + 4x^2 + 2x$, which has complex multiplication by $\bbZ[\sqrt{-2}]$ by~\cite[Proposition~II.2.3.1(ii)]{SilvermanAdv}. One computes that $\Hom_{\Gamma_K}(\cE_3,\cE^{-3}_3)^-\cong (\bbZ/3\bbZ)^2$ and $\Hom_{\Gamma_\bbQ}(\cE_3,\cE^{-3}_3)^-\cong \Hom_{\Gamma_\bbQ}(\cE_3,\cE^{6}_3)^- \cong \bbZ/3\bbZ$. This can be done by calculating the $3$-torsion points of $\cE$ and using the fact that the $\Gamma_F$-module $\Hom(\cE_3,\cE^{-3}_3)^-$ is the quadratic twist of $\End( \cE_3)^-$ by the quadratic character corresponding to $F(\sqrt{-3})/F$. Alternatively, for an explicit computation of the $\Gamma_\bbQ$-modules $\Hom(\cE_3,\cE^{a}_3)$ for $a\in\{-3,6\}$, see the proof of~\cite[Lemma~2.3.3]{MohamedThesis}.

Now Lemma~\ref{lem:3} and Corollary~\ref{cor:Br-} show that
\[\left(\frac{\Br(E\times E^a)}{\Br_1(E\times E^{a})}\right)_{3^\infty}= \frac{\Br(E\times E^a)_3}{\Br_1(E\times E^{a})_3}=\Hom_{\Gamma_F}(\cE_3,\cE^{a}_3)^-.\]
By Theorem~\ref{thm:quad}\ref{-2}, it only remains to show that the $2$-primary part of $\Br(E\times E^a)/\Br_1(E\times E^{a})$ is trivial. Write $A=E\times E^a$ and $L=K(\sqrt{a})$. By~\eqref{eq:L1},~\eqref{eq:n} and the computation of $n(2)$ in the proof of Theorem~\ref{thm:quad}, it is enough to show that $\Br(A_L)_2/\Br_1(A_L)_2=0$. Now Proposition~\ref{prop:SZhom} shows that \[\Br(A_L)_2/\Br_1(A_L)_2=\End_{\Gamma_L}(E_2)/(\End\bar{E}/2)^{\Gamma_L}.\]
One computes that this quotient is trivial. As before, one can take $E=\cE$ and compute the $2$-torsion explicitly. For the details, see the proof of~\cite[Lemma~2.3.6]{MohamedThesis}.

\paragraph{$\bullet\  K=\bbQ(\sqrt{-11})$:} We take $\cE$ to be the elliptic curve with LMFDB label 121.b2, which has affine equation $y^2+y=x^3-x^2-7x+10$. This is the modular curve $X_{\textrm{ns}}^+(11)$. It has complex multiplication by $\bbZ[\frac{1+\sqrt{-11}}{2}]$ by~\cite{LMFDB}. One computes that $\Hom_{\Gamma_K}(\cE_3,\cE^{-3}_3)^-\cong (\bbZ/3\bbZ)^2$ and $\Hom_{\Gamma_\bbQ}(\cE_3,\cE^{-3}_3)^-\cong \Hom_{\Gamma_\bbQ}(\cE_3,\cE^{33}_3)^- \cong \bbZ/3\bbZ$. Explicit calculations can be found in the proof of~\cite[Theorem~2.4.1]{MohamedThesis}. Now Theorem~\ref{thm:quad}\ref{-11}, Lemma~\ref{lem:3} and Corollary~\ref{cor:Br-} show that
\[\frac{\Br(E\times E^a)}{\Br_1(E\times E^{a})}= \frac{\Br(E\times E^a)_3}{\Br_1(E\times E^{a})_3}= \Hom_{\Gamma_F}(\cE_3,\cE^{a}_3)^-. \qedhere\]
\end{proof}

\section{CM by $\bbZ[\zeta_3]$: transcendental Brauer groups}\label{sec:zetagp}

Throughout this section, for $c\in\bbQ^\times$, let $E^c$ denote the elliptic curve over $\bbQ$ with affine equation
$$E^c:y^2=x^3+c.$$ 
The curve $E^c$ has complex multiplication by $\bbZ[\zeta_3]$, where $\zeta_3$ denotes a primitive $3$rd root of unity. Multiplication by $\zeta_3$ sends
\[(x,y)\mapsto (\zeta_3x, y). \]
The curve $E^c$ is the sextic twist of $y^2=x^3+1$ by the class of $c^{-1}$ in $\mathrm{H}^1(\bbQ,\mu_6)$. Since $\bbQ(\zeta_3)$ has class number one, any elliptic curve over $\bbQ$ with complex multiplication by $\bbZ[\zeta_3]$ is of the form $E^c$ for some $c\in\bbQ^\times$.

In this section, we study the transcendental Brauer groups of $E^c\times E^d$ and $\Kum(E^c\times E^d)$ for $c,d\in\bbQ^\times$.

\begin{lemma}
\label{lem:prime bound}
Let $c,d\in\bbQ^\times$.
For every prime number $\ell>7$, $$(\Br(E^c\times E^d)/\Br_1(E^c\times E^d))_{\ell^\infty}=0.$$ 
For $\ell\in\{5,7\}$, $$(\Br(E^c\times E^d)/\Br_1(E^c\times E^d))_{\ell^\infty}=\Br(E^c\times E^d)_{\ell}/\Br_1(E^c\times E^d)_{\ell}.$$
For $\ell\in\{2,3\}$,  $$(\Br(E^c\times E^d)/\Br_1(E^c\times E^d))_{\ell^\infty}=(\Br(E^c\times E^d)/\Br_1(E^c\times E^d))_{\ell^2}.$$
\end{lemma}

\begin{proof}
Let $A=E^c\times E^d$ and let $Y=\Kum A$.
By definition of $\Br_1 (Y)$, we have an injection
\[\Br (Y)/\Br_1 (Y)\hookrightarrow \Br (\bar{Y})^{\Gamma_{\bbQ(\zeta_3)}}.\]
Since $Y$ is a K3 surface with CM by $\bbZ[\zeta_3]$, we can apply \cite[Example~2, pp.50--51]{Valloni} to $Y$ to see that, for all primes $\ell>7$, $\Br (\bar{Y})^{\Gamma_{\bbQ(\zeta_3)}}_{\ell^\infty}=0$ and hence $(\Br (Y)/\Br_1 (Y))_{\ell^\infty}=0$. Since $\Br (Y)$ is a torsion group, $(\Br (Y)/\Br_1 (Y))_{\ell^\infty}=\Br (Y)_{\ell^\infty}/\Br_1 (Y)_{\ell^\infty}$. The first statement now follows easily from Theorem~\ref{thm:SZBrInject}. 

Let $L=\bbQ\bigl(\zeta_3, \sqrt[6]{c/d}\bigr)$ so that $A_L\cong E^c\times E^c$. By definition, 
 \begin{equation}\label{eq:L}
 \Br (A)_n/\Br_1(A)_n\hookrightarrow \Br(A_L)_n/\Br_1(A_L)_n
 \end{equation}
 for all $n\in\bbZ_{>0}$. By Theorem~\ref{thm:RN}, for any prime number $\ell$,
\begin{equation}\label{eq:m}
(\Br(A_L)/\Br_1(A_L))_{\ell^\infty}=\Br(A_L)_{\ell^{n(\ell)}}/\Br_1(A_L)_{\ell^{n(\ell)}}
\end{equation}
where $n(\ell)$ is as defined in Definition~\ref{def:n}.

Observe that $[L:\bbQ(\zeta_3)]\mid 6$. By~\eqref{eq:ringclass}, we have $[K_8:\bbQ(\zeta_3)]=4$ and $[K_{27}:\bbQ(\zeta_3)]=9$, whence $n(2), n(3)\leq 2$. This, together with~\eqref{eq:L} and~\eqref{eq:m}, proves the statement for $\ell\in\{2,3\}$. 

Now suppose that $\ell\in\{5,7\}$. By~\eqref{eq:ringclass}, $[K_{\ell^2}:\bbQ(\zeta_3)]=2\ell>6$, so $n(\ell)\leq 1$. Our discussion above shows that 
\begin{equation}\label{eq:m57}
(\Br(A_L)/\Br_1(A_L))_{\ell^\infty}=(\Br(A_L))_{\ell}/(\Br_1(A_L))_{\ell}.
\end{equation}
Now an argument using restriction and corestriction similar to the one used in the proof of Lemma~\ref{lem:3} shows that $(\Br(A)/\Br_1(A))_{\ell^\infty}=\Br(A)_{\ell}/\Br_1(A)_{\ell}$.
\end{proof}

In this paper, we will focus on the cases with CM by $\bbZ[\zeta_3]$ where the transcendental part of the Brauer group contains an element of order $5$ or $7$. The other cases will be discussed in future work.

\begin{lemma}\label{lem:repel}
Let $c,d\in\bbQ^\times$ and let $Y=\Kum(E^c\times E^d)$.
Let $\ell\in\{5,7\}$ and suppose that $(\Br (Y)/\Br_1 (Y))_\ell\neq 0$. Then~\eqref{eq:Br_inj} yields an isomorphism 
\[\Br (Y)/\Br_1 (Y)\cong \Br (E^c\times E^d)_\ell/\Br_1 (E^c\times E^d)_\ell.\]
\end{lemma}

\begin{proof}
This follows from~\cite[Example~2, pp.50--51]{Valloni} and Lemma~\ref{lem:prime bound}.
\end{proof}

We now calculate $\Br (E^c\times E^d)_\ell/\Br_1 (E^c\times E^d)_\ell$ for $\ell\in\{5,7\}$, using Corollary~\ref{cor:Br-}.
To compute $\Hom_{\Gamma_\bbQ}(E^d_\ell,E^c_\ell)^-$ for $\ell\in\{5,7\}$, we will use Eisenstein's sextic reciprocity law, as stated in~\cite[Theorem~7.10]{Lemmermeyer}.

\begin{definition}
An element $a+b\zeta_3\in\bbZ[\zeta_3]$ is called \emph{E-primary} if $b\equiv 0\bmod{3}$ and 
\begin{align*}
a+b\equiv 1\bmod{4},& \textrm{ if }2\mid b,\\
b\equiv 1\bmod{4},& \textrm{ if }2\mid a,\\
a\equiv 3\bmod{4}, & \textrm{ if } 2\nmid ab.
\end{align*}
\end{definition}
Let $N$ denote the norm map $N_{\bbQ(\zeta_3)/\bbQ}:\bbQ(\zeta_3)\to\bbQ$.
Recall the definition of the sextic residue symbol: for $\lambda,\pi\in\bbZ[\zeta_3]$ with $\pi$ prime, $\left(\frac{\lambda}{\pi}\right)_6$ is the unique $6$th root of unity satisfying
\[\lambda^{\frac{N(\pi)-1}{6}}\equiv \left(\frac{\lambda}{\pi}\right)_6\pmod{\pi}.\]

\begin{theorem}[Eisenstein]\label{thm:sextic}
If $\beta,\gamma\in\bbZ[\zeta_3]$ are E-primary and relatively prime, then
\[\left(\frac{\beta}{\gamma}\right)_6=(-1)^{\frac{N(\beta)-1}{2}\frac{N(\gamma)-1}{2}}\left(\frac{\gamma}{\beta}\right)_6.\]
\end{theorem}

\begin{definition}
Let $\alpha=\sqrt[6]{c/d}$ and let $\phi_{\alpha}:\bar{E}^d\to \bar{E}^c$ be the isomorphism defined over $\bbQ(\alpha)$ given by $(x,y)\mapsto (\alpha^2x, \alpha^3y)$. 
\end{definition}
We use $ \langle-3^3\rangle\bbQ^{\times 6}$ to denote $\bbQ^{\times 6}\cup-3^3\cdot \bbQ^{\times 6}$.
\begin{proposition}\label{prop:gen}
View $\Gamma_{\bbQ(\alpha)}$ and $\Gamma_{\bbQ(\zeta_3)}$ as subgroups of $\Gamma_\bbQ$, so that the set difference $\Gamma_{\bbQ(\alpha)}\setminus\Gamma_{\bbQ(\zeta_3)}$ is defined. 
\begin{enumerate}[label=(\roman*)]
\item If $2^4\cdot 5\cdot cd\in\langle-3^3\rangle\bbQ^{\times 6}$ then, for $\tau\in\Gamma_{\bbQ(\alpha)}\setminus\Gamma_{\bbQ(\zeta_3)}$, abusing notation and viewing $\tau\circ\phi_\alpha$ as an element of $\Hom(E^d_5,E^c_5)$,
\begin{equation*}
\Br (E^c\times E^d)_5/\Br_1 (E^c\times E^d)_5=\Hom_{\Gamma_\bbQ}(E^d_5,E^c_5)^-=(\bbZ/5\bbZ)\cdot \tau\circ\phi_\alpha
\cong \bbZ/5\bbZ.
\end{equation*}

Otherwise, $\Br (E^c\times E^d)_5/\Br_1 (E^c\times E^d)_5=0$.
\item If $-2^4\cdot 7^{-1}\cdot cd\in\langle-3^3\rangle\bbQ^{\times 6}$ then, for $\tau\in\Gamma_{\bbQ(\alpha)}\setminus\Gamma_{\bbQ(\zeta_3)}$, abusing notation and viewing $\tau\circ\phi_\alpha$ as an element of $\Hom(E^d_7,E^c_7)$,
\begin{equation*}
\Br (E^c\times E^d)_7/\Br_1 (E^c\times E^d)_7=\Hom_{\Gamma_\bbQ}(E^d_7,E^c_7)^-=(\bbZ/7\bbZ)\cdot \tau\circ\phi_\alpha
\cong \bbZ/7\bbZ.
\end{equation*}
Otherwise, $\Br (E^c\times E^d)_7/\Br_1 (E^c\times E^d)_7=0$.
\end{enumerate}
\end{proposition}

\begin{proof}
Multiplying by $6$th powers if necessary, we may assume that $c,d\in\bbZ$. By Corollary~\ref{cor:Br-} it suffices to compute $ \Hom_{\Gamma_\bbQ}(E^d_\ell,E^c_\ell)^-$ for $\ell\in\{5,7\}$. Let $\varepsilon(\ell)=(-1)^{(\ell-1)/2}$. First, we will show that 
\[
 \Hom_{\Gamma_{\bbQ(\zeta_3)}}(E^d_\ell,E^c_\ell)^-=\begin{cases}
  \Hom(E^d_\ell,E^c_\ell)^- & \textrm{ if } \varepsilon(\ell)\cdot 2^4 \ell^{\varepsilon(\ell)}\cdot cd \in \langle-3^3\rangle\bbQ^{\times 6};\\
  0 & \textrm{ otherwise.}
 \end{cases}
\]
To prove this claim, we will determine the action of $\Gamma_{\bbQ(\zeta_3)}$ on $ \Hom(E^d_\ell,E^c_\ell)^- $ via a study of the actions of Frobenius elements for sufficiently many primes in $\bbZ[\zeta_3]$. The action of $\Gamma_{\bbQ(\zeta_3)}$ factors through $\Gal(\bbQ(\zeta_3, E^c_\ell, E^d_\ell)/\bbQ(\zeta_3))$. 
Let $\pi\in\bbZ[\zeta_3]$ be an E-primary prime
that is coprime to $6cd\ell$ and unramified in $\bbQ(\zeta_3, E^c_\ell, E^d_\ell)/\bbQ(\zeta_3)$. The prime ideals generated by such $\pi$ comprise all but finitely many prime ideals of $\bbZ[\zeta_3]$, as every prime ideal of $\bbZ[\zeta_3]$ that is coprime to $6$ has an E-primary generator. Furthermore, Chebotarev's density theorem shows that $\Gal(\bbQ(\zeta_3, E^c_\ell, E^d_\ell)/\bbQ(\zeta_3))$ is generated by Frobenius elements associated to such primes $\pi$. We require $\pi$ to be E-primary so that we can apply sextic reciprocity later on in the proof. In particular, we have $\pi\equiv \pm 1\pmod{3}$.

For $a\in\{c,d\}$, let $\psi_{E^a/\bbQ(\zeta_3)}$ be the Gr\"{o}ssencharakter attached to $E^a/\bbQ(\zeta_3)$. We write $\psi_{E^a/\bbQ(\zeta_3)}(\pi)$ for the image under $\psi_{E^a/\bbQ(\zeta_3)}$ of the id\`{e}le $(1,\ldots,1,\pi, 1,1,\ldots)$ with entry $\pi$ at the place $(\pi)$ and entry $1$ at every other place. Then $\Frob_\pi\in\Gal(\bbQ(\zeta_3, E^c_\ell, E^d_\ell)/\bbQ(\zeta_3))$ acts on $E^a_\ell$ as multiplication by $\psi_{E^a/\bbQ(\zeta_3)}(\pi)\in\bbZ[\zeta_3]$, see~\cite[Corollary~4.1.3]{Lang}, for example. By~\cite[Example II.10.6]{SilvermanAdv}, 
\begin{equation}\label{eq:grossencharpi}
\psi_{E^a/\bbQ(\zeta_3)}(\pi)=\pm\overline{\left(\frac{4a}{\pi}\right)}_6\pi=\pm\left(\frac{4a}{\pi}\right)^{-1}_6\pi.
\end{equation}
The $\pm 1$ here comes from the fact that~\cite[Example II.10.6]{SilvermanAdv} is stated for primes that are congruent to $-1\pmod{3}$, whereas our E-primary prime $\pi$ maybe congruent to either $1$ or $-1\pmod{3}$. In any case, the $\pm 1$ in~\eqref{eq:grossencharpi} is independent of $a$ and is therefore of no consequence for the action on $\Hom(E^d_\ell,E^c_\ell)^-$ by conjugation, as the $\pm 1$ for the actions on $E^c$ and $E^d$ cancel out.
Thus, for $\varphi\in \Hom(E^d_\ell,E^c_\ell)^-$, we have
\begin{align}\label{eq:conjFrob}
\Frob_\pi\cdot \varphi&=\left(\frac{4c}{\pi}\right)^{-1}_6\pi \varphi \left(\frac{4d}{\pi}\right)_6\pi^{-1}\nonumber\\
&= \left(\frac{2^4cd}{\pi}\right)^{-1}_6\pi \bar{\pi}^{-1}\varphi.
\end{align}

$\bullet$ For $\ell=5$, sextic reciprocity gives
\[\left(\frac{5}{\pi}\right)_6=\left(\frac{\pi}{5}\right)_6\equiv \pi^4\pmod{5}.\]
Furthermore, $\bar{\pi}\equiv \pi^5\pmod{5}$, whereby $\pi\bar{\pi}^{-1}\equiv \left(\frac{5}{\pi}\right)_6^{-1}\pmod{5}$. Substituting this into~\eqref{eq:conjFrob} gives 
\begin{equation}\label{eq:num5}
\Frob_\pi\cdot \varphi=\left(\frac{2^4\cdot 5 \cdot cd}{\pi}\right)^{-1}_6\varphi
\end{equation}
for all E-primary primes $\pi\in\bbZ[\zeta_3]$ that are coprime to $30cd$ and unramified in $\bbQ(\zeta_3, E^c_5, E^d_5)/\bbQ(\zeta_3)$. Since $z-1$ is invertible modulo $5$ for all $z\in\mu_6\setminus\{1\}$, we deduce that $\Frob_\pi\cdot \varphi=\varphi$ if and only if either $\varphi=0$ or $\left(\frac{2^4\cdot 5 \cdot cd}{\pi}\right)_6=1$ for all E-primary primes $\pi\in\bbZ[\zeta_3]$ that are coprime to $30cd$ and unramified in $\bbQ(\zeta_3, E^c_5, E^d_5)/\bbQ(\zeta_3)$. The latter condition holds if and only if
$2^4\cdot 5\cdot cd\in\bbQ^\times\cap \bbQ(\zeta_3)^{\times 6}=\langle-3^3\rangle\bbQ^{\times 6}$ (see~\cite[Theorem~9.1.11]{Neukirch}). Hence, if $ \Hom_{\Gamma_{\bbQ(\zeta_3)}}(E^d_5,E^c_5)^-\neq 0$ then $2^4\cdot 5\cdot cd\in\bbQ^\times\cap \bbQ(\zeta_3)^{\times 6}$, whereby~\eqref{eq:num5} shows that $\Hom_{\Gamma_{\bbQ(\zeta_3)}}(E^d_5,E^c_5)^-=\Hom(E^d_5,E^c_5)^-$, as required.

$\bullet$ For $\ell=7$, factorise $7$ in $\bbZ[\zeta_3]$ as $7=\varpi\bar{\varpi}$ with $\varpi=-1-3\zeta_3$. Then $\varpi$ and $-\bar{\varpi}$ are both E-primary and sextic reciprocity gives
\begin{equation}\label{eq:7recip}
\left(\frac{-7}{\pi}\right)_6=\left(\frac{\varpi}{\pi}\right)_6\left(\frac{-\bar{\varpi}}{\pi}\right)_6=\left(\frac{\pi}{\varpi}\right)_6\left(\frac{\pi}{\bar{\varpi}}\right)_6\equiv \pi\bar{\pi}^{-1}\pmod{\varpi}.
\end{equation}
Taking complex conjugates and then inverting both sides of~\eqref{eq:7recip} gives $\left(\frac{-7}{\pi}\right)_6\equiv \pi\bar{\pi}^{-1}\pmod{\bar{\varpi}}$ and hence $\left(\frac{-7}{\pi}\right)_6\equiv \pi\bar{\pi}^{-1}\pmod{7}$. Substituting this into~\eqref{eq:conjFrob} gives 
\begin{equation}\label{eq:num7}
\Frob_\pi\cdot \varphi=\left(\frac{-2^4\cdot 7^{-1}\cdot cd}{\pi}\right)^{-1}_6\varphi
\end{equation}
for all E-primary primes $\pi\in\bbZ[\zeta_3]$ that are coprime to $42cd$ and unramified in $\bbQ(\zeta_3, E^c_7, E^d_7)/\bbQ(\zeta_3)$. 
As before, we deduce that if $ \Hom_{\Gamma_{\bbQ(\zeta_3)}}(E^d_7,E^c_7)^-\neq 0$ then $-2^4\cdot 7^{-1}\cdot cd\in\bbQ^\times\cap \bbQ(\zeta_3)^{\times 6}=\langle-3^3\rangle\bbQ^{\times 6}$ and $ \Hom_{\Gamma_{\bbQ(\zeta_3)}}(E^d_7,E^c_7)^-= \Hom(E^d_7,E^c_7)^-$, completing the proof of our claim.

To complete the proof of Proposition~\ref{prop:gen}, it remains to compute $ \Hom_{\Gamma_\bbQ}(E^d_\ell,E^c_\ell)^-$ for $\ell\in\{5,7\}$ in the case where $\varepsilon(\ell)\cdot 2^4 \ell^{\varepsilon(\ell)}\cdot cd \in \langle-3^3\rangle\bbQ^{\times 6}$. It is easy to see that the conditions on $cd$ ensure that $\bbQ(\zeta_3)\not\subset\bbQ(\alpha)$ and hence there exists some $\tau\in \Gamma_{\bbQ(\alpha)}\setminus \Gamma_{\bbQ(\zeta_3)}$.
Now observe that 
\[\Hom_{\Gamma_{\bbQ(\zeta_3)}}(E^d_\ell,E^c_\ell)^-=\Hom(E^d_\ell,E^c_\ell)^-=\{(a+b\zeta_3)\tau\circ\phi_\alpha\mid a,b\in \bbZ/\ell\bbZ\}.\]
Indeed, it is clear that $\{(a+b\zeta_3)\tau\circ\phi_\alpha\mid a,b\in \bbZ/\ell\bbZ\}\subset \Hom(E^d_\ell,E^c_\ell)^-$ and both are isomorphic to $(\bbZ/\ell\bbZ)^2$ as abelian groups. 

Furthermore, since the image of $\tau$ generates $\Gal(\bbQ(\zeta_3)/\bbQ)$, an element of $\Hom_{\Gamma_{\bbQ(\zeta_3)}}(E^d_\ell,E^c_\ell)^-$ is fixed by the action of $\Gamma_{\bbQ}$ if and only if it commutes with $\tau$. Therefore, $\Hom_{\Gamma_{\bbQ}}(E^d_\ell,E^c_\ell)^-=(\bbZ/\ell\bbZ)\cdot\tau\circ\phi_\alpha$, as claimed.
\end{proof}

\begin{proof}[Proof of Theorem~\ref{thm:intro57}]
This now follows from Theorem~\ref{thm:SZBrInject}, 
Proposition~\ref{prop:gen} and Lemma~\ref{lem:repel}.
\end{proof}

\section{Algebraic Brauer groups}\label{sec:Br1}

Let $E$ and $E'$ be elliptic curves over $\bbQ$ and let $Y=\Kum(E\times E')$. 
Since $Y(\bbQ)\neq \emptyset$, the Hochschild--Serre spectral sequence gives a short exact sequence
\begin{equation}\label{eq:H-S}
0\to \Br(\bbQ)\to\Br_1(Y)\to \mathrm{H}^1(\bbQ, \Pic(\bar{Y}))\to 0.
\end{equation}
Since $Y$ is a K3 surface, $\Pic(\bar{Y})=\NS(\bar{Y})$. Furthermore,~\cite[Proposition~1.4(i)]{SZtorsion} gives a short exact sequence 
\begin{equation}\label{eq:NS}
0\to N_{\Lambda}\oplus N_{\Sigma}\to \NS(\bar{Y}) \to  \Hom(\bar{E},\bar{E}')\to 0,
\end{equation}
where $N_{\Lambda}$ and $N_{\Sigma}$ are permutation $\Gamma_\bbQ$ modules and hence $\mathrm{H}^1(\bbQ,N_{\Lambda})=\mathrm{H}^1(\bbQ, N_{\Sigma})=0$. Therefore, the long exact sequence of Galois cohomology attached to~\eqref{eq:NS} can be combined with~\eqref{eq:H-S} to yield
\begin{equation}\label{eq:PicHom1}
 \Br_1(Y)/\Br(\bbQ)=\mathrm{H}^1(\bbQ, \Pic(\bar{Y}))=\mathrm{H}^1(\bbQ, \Hom(\bar{E},\bar{E}')).
\end{equation}

Now suppose that $\End\bar{E}=\cO_K$ for an imaginary quadratic field $K$ and suppose that there exists a finite extension $L/K$ such that $L/\bbQ$ is Galois and $E_L\cong E'_L$. Then inflation-restriction gives an exact sequence
\begin{equation}\label{eq:infresNS}
0\to \mathrm{H}^1(\Gal(L/\bbQ), \Hom(\bar{E},\bar{E}'))\to \mathrm{H}^1(\bbQ, \Hom(\bar{E},\bar{E}'))\to \mathrm{H}^1(L, \Hom(\bar{E},\bar{E}')), 
\end{equation}
where we view $\Hom(\bar{E},\bar{E}')$ as a twist of $\cO_K$ defined over $L$. But then $\mathrm{H}^1(L, \Hom(\bar{E},\bar{E}'))\cong\Hom_{\textrm{cts}}(\Gamma_L, \bbZ^2)=0$.
Thus,~\eqref{eq:infresNS} gives a canonical isomorphism from $\mathrm{H}^1(\Gal(L/\bbQ), \Hom(\bar{E},\bar{E}'))$ to $\mathrm{H}^1(\bbQ, \Hom(\bar{E},\bar{E}'))$. Plugging this into~\eqref{eq:PicHom1} gives
\begin{equation}\label{eq:Pichom}
 \Br_1(Y)/\Br(\bbQ)=\mathrm{H}^1(\bbQ, \Hom(\bar{E},\bar{E}'))=\mathrm{H}^1(\Gal(L/\bbQ), \Hom(\bar{E},\bar{E}')).
\end{equation}

\begin{theorem}\label{thm:Br1-2} 
Let $d\in\bbZ_{>0}$ be squarefree so that $K=\bbQ(\sqrt{-d})$ is an imaginary quadratic field. Let $E/\bbQ$ be an elliptic curve with $\End\bar{E}=\cO_K$, let $E'/\bbQ$ be the quadratic twist of $E$ by $a\in\bbQ^\times$ and let $Y=\Kum(E\times E')$. Then 
\[\Br_1(Y)/\Br(\bbQ)\cong\begin{cases}
0 & \textrm{ if } \bbQ(\sqrt{a})\subset K\textrm{ and } -d\equiv 1\bmod{4};\\
(\bbZ/2\bbZ)^2 & \textrm{ if } \bbQ(\sqrt{a})\not\subset K \textrm{ and } -d\equiv 2,3 \bmod{4};\\
\bbZ/2\bbZ &\textrm{ otherwise.}\\
\end{cases}
\]
\end{theorem}

\begin{proof}
Let $L=K(\sqrt{a})$. Then~\eqref{eq:Pichom} gives
\begin{equation*}
\Br_1(Y)/\Br(\bbQ)=\mathrm{H}^1(\Gal(L/\bbQ), \Hom(\bar{E},\bar{E}'))
\end{equation*}
where the $\Gamma_\bbQ$-module $ \Hom(\bar{E},\bar{E}')$ is the twist of $\End\bar{E}=\cO_K$ by the quadratic character corresponding to $\bbQ(\sqrt{a})/\bbQ$. 

If $\bbQ(\sqrt{a})\subset K$ then $\Gal(L/\bbQ)=\Gal(K/\bbQ)$ is cyclic and a Tate cohomology calculation gives 
\[\mathrm{H}^1(\Gal(L/\bbQ), \Hom(\bar{E},\bar{E}'))\cong\begin{cases}
0 & \textrm{ if } -d\equiv 1\bmod{4};\\
\bbZ/2\bbZ & \textrm{ if } -d\equiv 2,3\bmod{4}.
\end{cases}\]

If $\bbQ(\sqrt{a})\not\subset K$ then letting $G=\Gal(L/\bbQ)$, $N=\Gal(L/K)$ and $M=\Hom(\bar{E},\bar{E}')$, inflation-restriction gives an exact sequence
\[0\to \mathrm{H}^1(G/N, M^N)\to \mathrm{H}^1(G,M)\to \mathrm{H}^1(N, M)^{G/N}\to \mathrm{H}^2(G/N, M^N).\]
Since $M$ is the twist of $\cO_K$ by the character corresponding to $\bbQ(\sqrt{a})/\bbQ$, the generator of $N$ acts as multiplication by $-1$ on $M$, whereby $M^N=0$ and the inflation-restriction sequence yields an isomorphism from $ \mathrm{H}^1(G,M)$ to $\mathrm{H}^1(N, M)^{G/N}$. Now a Tate cohomology calculation gives $\mathrm{H}^1(N, M)=M/2M$. One checks that 
\[(M/2M)^{G/N}\cong \begin{cases}
\bbZ/2\bbZ& \textrm{ if } -d\equiv 1\bmod{4}\\
(\bbZ/2\bbZ)^2& \textrm{ if } -d\equiv 2,3\bmod{4}. \qedhere
\end{cases}\]
\end{proof}

We now have all the necessary ingredients for the proof of Theorem~\ref{thm:mainquad}.

\begin{proof}[Proof of Theorem~\ref{thm:mainquad}]
By Theorem~\ref{thm:Br1-2}, $\Br_1(Y)\setminus\Br(\bbQ)$ contains no elements of odd order. Therefore, the assumption that $\Br(Y)\setminus \Br(\bbQ)$ contains an element of odd order implies that $\Br(Y)/\Br_1(Y)\neq 0$. 
Therefore, by Theorem~\ref{thm:SZBrInject} and Theorem~\ref{thm:quad}, $\Br(Y)/\Br_1(Y)$ contains an element of order $3$ and $K\in\{\bbQ(\sqrt{-2}),\bbQ(\sqrt{-11})\}$, proving~\eqref{Kquadintro}. Hence, $Y=\Kum(E\times E^a)$ with $a\in -3K^{\times 2}\cap\bbQ^\times$, by Proposition~\ref{prop:alpha}. Now~\eqref{trBrquadintro} follows from Theorem~\ref{thm:SZBrInject} and Theorem~\ref{thm:numerator3}, and~\eqref{algBrquadintro} follows from Theorem~\ref{thm:Br1-2}. To prove~\eqref{Yquadintro}, note that $K$ has class number one and therefore $E$ 
is a quadratic twist of any chosen elliptic curve $\cE/\bbQ$ with $\End\bar{\cE}=\cO_K$. We take $\cE$ with affine equation $y^2=f(x)$, where $f(x)$ is as stated in Theorem~\ref{thm:mainquad}. Thus, $E$ has equation $\lambda y^2=f(x)$ for some $\lambda\in\bbQ^\times$ and $\Kum(E\times E^a)$ is the minimal desingularisation of the projective surface with affine equation $\lambda^2u^2=af(x)f(t)$. Replacing $u$ by $\lambda u$ and computing $(-3K^{\times 2}\cap\bbQ^\times)/\bbQ^{\times 2}$ completes the proof.
\end{proof}

\begin{theorem}\label{thm:Br1} For $a\in\bbQ$, let $E^a/\bbQ$ be the elliptic curve with affine equation $y^2=x^3+a$.
Let $c,d\in\bbZ\setminus\{0\}$, let $Y=\Kum(E^c\times E^d)$ and let $\alpha=\sqrt[6]{c/d}$.
We have
\[\Br_1(Y)/\Br(\bbQ)\cong\begin{cases}
\bbZ/3\bbZ & \textrm{ if } \bbQ(\zeta_3)\subset \bbQ(\alpha) \textrm{ and } [\bbQ(\alpha):\bbQ(\zeta_3)]=3;\\
\bbZ/2\bbZ & \textrm{ if } [\bbQ(\alpha):\bbQ]=2 \textrm{ and }\bbQ(\alpha)\neq\bbQ(\zeta_3);\\
0 & \textrm{ otherwise.}
\end{cases}
\]
\end{theorem}

\begin{proof}
If $[\bbQ(\alpha):\bbQ]\leq 2$ then $E^d$ is a quadratic twist of $E^c$ and the result follows from Theorem~\ref{thm:Br1-2}. So henceforth we may assume that $[\bbQ(\alpha):\bbQ]\in\{3,6\}$.

Let $\phi_{\alpha}:\bar{E}^d\to \bar{E}^c$ be the isomorphism defined over $\bbQ(\alpha)$ given by $(x,y)\mapsto (\alpha^2x, \alpha^3y)$, whereby
\begin{equation*}
\Hom(\bar{E}^d,\bar{E}^c)=\End(\bar{E}^c)\circ\phi_\alpha=\bbZ[\zeta_3]\circ\phi_\alpha
\end{equation*}
so~\eqref{eq:Pichom} gives
\begin{equation}\label{eq:infBr1}
\Br_1(Y)/\Br(\bbQ)=\mathrm{H}^1(\Gal(L/\bbQ),\bbZ[\zeta_3]\circ\phi_\alpha)
\end{equation}
where $L=\bbQ(\zeta_3,\alpha)$. It remains to calculate $\mathrm{H}^1(\Gal(L/\bbQ), \bbZ[\zeta_3]\circ\phi_\alpha)$.

Inflation-restriction gives an exact sequence 
\begin{align}\label{eq:inf-res}
&0\to\mathrm{H}^1(\Gal(\bbQ(\zeta_3)/\bbQ),(\bbZ[\zeta_3]\circ\phi_\alpha)^{\Gal(L/\bbQ(\zeta_3))}) \to \mathrm{H}^1(\Gal(L/\bbQ),\bbZ[\zeta_3]\circ\phi_\alpha)\\
& \to \mathrm{H}^1(\Gal(L/\bbQ(\zeta_3)),\bbZ[\zeta_3]\circ\phi_\alpha)^{\Gal(\bbQ(\zeta_3)/\bbQ)} \to \mathrm{H}^2(\Gal(\bbQ(\zeta_3)/\bbQ),(\bbZ[\zeta_3]\circ\phi_\alpha)^{\Gal(L/\bbQ(\zeta_3))}).\nonumber
\end{align}

Since $\bbQ(\alpha)\not\subset \bbQ(\zeta_3)$, the Galois group $\Gal(L/\bbQ(\zeta_3))$ acts non-trivially on $\alpha$ and therefore on $\phi_\alpha$, and we have $(\bbZ[\zeta_3]\circ\phi_\alpha)^{\Gal(L/\bbQ(\zeta_3))}=0$. Thus, \eqref{eq:inf-res} yields
\begin{align}\label{eq:inf-res3}
\mathrm{H}^1(\Gal(L/\bbQ),\bbZ[\zeta_3]\circ\phi_\alpha) = \mathrm{H}^1(\Gal(L/\bbQ(\zeta_3)),\bbZ[\zeta_3]\circ\phi_\alpha)^{\Gal(\bbQ(\zeta_3)/\bbQ)}.
\end{align}
Since $\Gal(L/\bbQ(\zeta_3))$ is cyclic, 
we can use Tate cohomology to compute $ \mathrm{H}^1(\Gal(L/\bbQ(\zeta_3)),\bbZ[\zeta_3]\circ\phi_\alpha)$.

First suppose that $\Gal(L/\bbQ(\zeta_3))\cong \bbZ/3\bbZ$. Then we can choose a generator $\sigma$ of $\Gal(L/\bbQ(\zeta_3 ))$ that sends $\phi_{\alpha}$ to $\zeta_3 \phi_{\alpha}$. Therefore,
\begin{align}\label{eq:3}
\mathrm{H}^1(\Gal(L/\bbQ(\zeta_3)),\bbZ[\zeta_3]\circ\phi_\alpha)&\cong (\bbZ[\zeta_3]\circ\phi_\alpha)/(\zeta_3-1) (\bbZ[\zeta_3]\circ\phi_\alpha)\\
&\cong(\bbZ/3\bbZ)\cdot\phi_\alpha.\nonumber
\end{align}
The first isomorphism in~\eqref{eq:3} is induced by sending a $1$-cocycle to its value at $\sigma$. Let $f: \Gal(L/\bbQ(\zeta_3))\to \bbZ[\zeta_3]\circ\phi_\alpha$ be the $1$-cocycle sending $\sigma$ to $\phi_\alpha$. To determine whether $\mathrm{H}^1(\Gal(L/\bbQ(\zeta_3)),\bbZ[\zeta_3]\circ\phi_\alpha)^{\Gal(\bbQ(\zeta_3)/\bbQ)}$ is trivial or isomorphic to $\bbZ/3\bbZ$, we just have to check whether the class of $f$ in $\mathrm{H}^1(\Gal(L/\bbQ(\zeta_3)),\bbZ[\zeta_3]\circ\phi_\alpha)$ is fixed by the action of $\Gal(\bbQ(\zeta_3)/\bbQ)$. Let $\tau\in \Gal(L/\bbQ)$ be such that its image in $\Gal(\bbQ(\zeta_3)/\bbQ)$ is non-trivial. Then $\tau\sigma\tau^{-1}=\sigma^{-1}$ and the $1$-cocycle property gives 
\[f(\tau^{-1}\sigma\tau)=f(\sigma^{-1})=-\sigma^{-1}(f(\sigma))=-\zeta_3^2\phi_\alpha.\]
Therefore, 
\begin{equation}\label{eq:tau}
f^\tau(\sigma)=\tau\cdot f(\tau^{-1}\sigma\tau)=\tau\cdot(-\zeta_3^2\phi_\alpha)=-\zeta_3\tau\cdot \phi_\alpha,
\end{equation}
by definition of the action of $\tau$ on $\mathrm{H}^1(\Gal(L/\bbQ(\zeta_3)),\bbZ[\zeta_3]\circ\phi_\alpha)$ in terms of its actions on $\Gal(L/\bbQ(\zeta_3))$ and on $\bbZ[\zeta_3]\circ\phi_\alpha=\Hom(\bar{E}^d,\bar{E}^c)$.

If $[\bbQ(\alpha):\bbQ]=3$ then we may assume that $\tau$ acts trivially on $\phi_\alpha$, whereby~\eqref{eq:tau} gives $f^\tau(\sigma)=-\zeta_3\phi_\alpha$. Hence, $f^\tau$ and $f$ are not cohomologous (as can be seen using~\eqref{eq:3}, for example). Therefore, $\mathrm{H}^1(\Gal(L/\bbQ(\zeta_3)),\bbZ[\zeta_3]\circ\phi_\alpha)^{\Gal(\bbQ(\zeta_3)/\bbQ)}=0$ and
\eqref{eq:infBr1} and \eqref{eq:inf-res3} give 
\begin{align*}
 \Br_1(Y)/\Br(\bbQ)&= \mathrm{H}^1(\Gal(L/\bbQ(\zeta_3)),\bbZ[\zeta_3]\circ\phi_\alpha)^{\Gal(\bbQ(\zeta_3)/\bbQ)}=0.
\end{align*}

On the other hand, if $[\bbQ(\alpha):\bbQ]=6$ then $\bbQ(\zeta_3)\subset \bbQ(\alpha)=L$ and in fact $\bbQ(\zeta_3)=\bbQ(\alpha^3)$. In this case we may assume that $\tau(\alpha)=-\alpha$ and hence $\tau\cdot \phi_\alpha=-\phi_\alpha$. Therefore,~\eqref{eq:tau} gives $f^\tau(\sigma)=\zeta_3\phi_\alpha$ and hence $f^\tau$ is cohomologous to $f$ (as can be seen using~\eqref{eq:3}, for example). Therefore, $\mathrm{H}^1(\Gal(L/\bbQ(\zeta_3)),\bbZ[\zeta_3]\circ\phi_\alpha)^{\Gal(\bbQ(\zeta_3)/\bbQ)}\cong (\bbZ/3\bbZ)\cdot\phi_\alpha$ and the result follows from \eqref{eq:infBr1} and \eqref{eq:inf-res3}.

Finally, suppose that $\Gal(L/\bbQ(\zeta_3))\cong \bbZ/6\bbZ$. Then we can choose a generator of $\Gal(L/\bbQ(\zeta_3))$ that sends $\phi_\alpha$ to $-\zeta_3\phi_\alpha$. Therefore,
\begin{align*}
\mathrm{H}^1(\Gal(L/\bbQ(\zeta_3)),\bbZ[\zeta_3]\circ\phi_\alpha)&\cong (\bbZ[\zeta_3]\circ\phi_\alpha)/(\zeta_3+1)(\bbZ[\zeta_3]\circ\phi_\alpha)=0.
\end{align*}
By \eqref{eq:infBr1} and \eqref{eq:inf-res3}, our proof is complete.
\end{proof}

Theorem~\ref{thm:introtransob} now follows easily from Theorem~\ref{thm:Br1}.

\begin{proof}[Proof of Theorem~\ref{thm:introtransob}]
We have $\varepsilon(\ell)\cdot 2^4\cdot \ell^{\varepsilon(\ell)}\cdot cd=(-3)^{3n}\cdot t^6$ for some $n\in\{0,1\}$ and $t\in\bbQ^\times$ so 
\[\frac{c}{d}=\frac{\varepsilon(\ell)\cdot2^4\cdot \ell^{\varepsilon(\ell)}\cdot c^2}{(-3)^{3n}\cdot t^6}.\]
Suppose for contradiction that $\bbQ(\zeta_3)\subset \bbQ(\sqrt[6]{c/d})$. Then $\bbQ(\zeta_3)=\bbQ(\sqrt{c/d})$ and hence $c/d\in-3\cdot\bbQ^{\times 2}$, which is evidently not the case.
Furthermore, $[\bbQ(\sqrt[6]{c/d}):\bbQ]=2$ if and only if $c/d$ lies in $\bbQ^{\times 3}$, if and only if $c\in 2\cdot \ell^{\varepsilon(\ell)}\cdot \bbQ^{\times 3}$. Now the result follows from Theorem~\ref{thm:Br1}.
\end{proof}

For completeness, we also include the calculation of the algebraic part of the Brauer group in the case of CM by $\bbZ[i]$.

\begin{theorem}\label{thm:Br1i} For $a\in\bbQ$, let $E^a/\bbQ$ be the elliptic curve with affine equation $y^2=x^3-ax$.
Let $c,d\in\bbZ\setminus\{0\}$, let $Y=\Kum(E^c\times E^d)$ and let $\alpha=\sqrt[4]{c/d}$.
We have
\[\Br_1(Y)/\Br(\bbQ)\cong\begin{cases}
(\bbZ/2\bbZ)^2 & \textrm{ if } [\bbQ(\alpha):\bbQ]=2 \textrm{ and }\bbQ(\alpha)\neq\bbQ(i);\\
\bbZ/2\bbZ & \textrm{ otherwise.}
\end{cases}
\]
\end{theorem}

\begin{proof}
The proof is very similar to that of Theorem~\ref{thm:Br1} so we shall be brief.
Given Theorem~\ref{thm:Br1-2}, our concern is the case $[\bbQ(\alpha):\bbQ]=4$. Let $\phi_{\alpha}:\bar{E}^d\to \bar{E}^c$ be the isomorphism defined over $\bbQ(\alpha)$ given by $(x,y)\mapsto (\alpha^2x, \alpha^3y)$, whereby
\begin{equation*}\label{eq:infBr12}
\Br_1(Y)/\Br(\bbQ)=\mathrm{H}^1(\Gal(L/\bbQ),\bbZ[i]\circ\phi_\alpha)
\end{equation*}
where $L=\bbQ(i,\alpha)$. Inflation-restriction gives 
\begin{align*}\label{eq:inf-res32}
\mathrm{H}^1(\Gal(L/\bbQ),\bbZ[i]\circ\phi_\alpha) = \mathrm{H}^1(\Gal(L/\bbQ(i)),\bbZ[i]\circ\phi_\alpha)^{\Gal(\bbQ(i)/\bbQ)}\cong\bbZ/2\bbZ. 
\end{align*}
\end{proof}

\section{Constant evaluation maps}\label{sec:pair_const}

In this section, we prove the following analogue of~\cite[Theorem~1.2(i)]{I-S}.

\begin{theorem}\label{thm:const}
Let $K\neq \bbQ(\sqrt{-11})$ be an imaginary quadratic field and 
let $Y=\Kum(E\times E')$ for elliptic curves $E,E'$ over $\bbQ$ with $\End \bar{E}=\End\bar{E}'=\cO_K$. Let $\ell$ be an odd prime and if $K=\bbQ(\zeta_3)$ assume that $\ell>3$. 
Suppose that $\cA\in\Br(Y)_{\ell}$ and let $v\neq \ell$ be a place of $\bbQ$. 
Then the evaluation map $\ev_{\cA,v}:Y(\bbQ_v)\to \Br(\bbQ_v)_\ell$ is constant.
\end{theorem}

The statement for odd primes of good reduction for $Y$ follows immediately from~\cite[Theorem D, Remark~1.6.2]{BNwild}, and in the remaining cases the proof boils down to showing $\ell$-divisibility for the sets of local points of certain elliptic curves with bad reduction. We will need the following elementary lemmas. 

\begin{lemma}\label{lem:varphigen}
Let $K$ be an imaginary quadratic field and 
let $Y=\Kum(E\times E')$ for elliptic curves $E,E'$ over $\bbQ$ with $\End \bar{E}=\End\bar{E}'=\cO_K$. Let $\ell$ be an odd prime and if $K=\bbQ(\zeta_3)$ assume that $\ell>3$. 
Suppose that $\cA\in\Br(Y)_{\ell}\setminus \Br(\bbQ)$.
Then $\ell\leq 7$, $K$ is one of $\bbQ(\zeta_3), \bbQ(i),\bbQ(\sqrt{-2}),\bbQ(\sqrt{-11})$ and
\begin{equation}\label{eq:Brhom}
\Br(Y)_\ell/\Br(\bbQ)_\ell=\Br(Y)_\ell/\Br_1(Y)_\ell=\Hom_{\Gamma_\bbQ}(E'_\ell, E_\ell)^- \cong\bbZ/\ell\bbZ.
\end{equation}
Furthermore, if $\varphi\in \Hom_{\Gamma_\bbQ}(E'_\ell, E_\ell)^- $ is non-zero then it is an isomorphism.
\end{lemma}
\begin{proof}
Suppose for contradiction that $\cA\in\Br_1(Y)$. Then Theorems~\ref{thm:Br1-2}, \ref{thm:Br1} and~\ref{thm:Br1i} show that $K=\bbQ(\zeta_3)$ and $\ell=3$, contradicting our assumptions. Consequently, $\cA\in \Br(Y)_\ell\setminus\Br_1(Y)$ and
Theorem~\ref{thm:mainoptions} shows that $K\in \{\bbQ(\zeta_3), \bbQ(i),\bbQ(\sqrt{-7}),\bbQ(\sqrt{-2}),\bbQ(\sqrt{-11})\}$ and $\ell\leq 7$. Since $\ell$ is odd, Theorem~\ref{thm:mainquad} shows that $K\neq \bbQ(\sqrt{-7})$.

For $K=\bbQ(\zeta_3)$,~\eqref{eq:Brhom} follows from Theorem~\ref{thm:Br1}, Theorem~\ref{thm:SZBrInject} and Proposition~\ref{prop:gen}.
For $K=\bbQ(i)$,~\eqref{eq:Brhom} follows from~\eqref{ISgp}
and~\cite[Section~4]{I-S}. For $K$ equal to $\bbQ(\sqrt{-2})$ or $\bbQ(\sqrt{-11})$,~\eqref{eq:Brhom} follows from Theorem~\ref{thm:mainquad}, Theorem~\ref{thm:SZBrInject} and Corollary~\ref{cor:Br-} (note that $\ell$ must equal $3$ in this case). 

Now let $0\neq \varphi\in \Hom_{\Gamma_\bbQ}(E'_\ell, E_\ell)^- $. We claim that $\varphi$ is an isomorphism. For $K=\bbQ(\zeta_3)$, this follows from the explicit generator given in Proposition~\ref{prop:gen}. For $K=\bbQ(i)$, it is proved in~\cite[Section~5.1]{I-S}, and we adapt the argument therein for $K$ equal to $\bbQ(\sqrt{-2})$ or $\bbQ(\sqrt{-11})$, observing that $3$ splits in $K/\bbQ$ as $3=\lambda\bar{\lambda}$ and we can write $E'_3=E'_\lambda\oplus E'_{\bar{\lambda}}$ but neither factor is a $\Gamma_\bbQ$-submodule of $E'_3$ and therefore neither factor can be the kernel of $\varphi$. Thus, the restrictions of $\varphi$ to $E'_\lambda$ and $E'_{\bar{\lambda}}$ are isomorphisms $E'_\lambda\to E_{\bar{\lambda}}$ and $E'_{\bar{\lambda}}\to E_{\lambda}$. Hence, $\varphi$ is an isomorphism.
\end{proof}

\begin{lemma}\label{lem:div-2}
Let $E/\bbQ$ be an elliptic curve with $\End\bar{E}=\bbZ[\sqrt{-2}]$. Then $2$ is a prime of bad reduction for $E$. Furthermore, if $p\neq 3$ is a prime of bad reduction for $E$ then $E(\bbQ_p)$ is $3$-divisible. 
\end{lemma}

\begin{proof}
Since $\bbQ(\sqrt{-2})$ has class number one, $E$ is a quadratic twist of $y^2=x^3+4x^2+2x$ by some squarefree $a\in\bbZ\setminus\{0\}$. Hence $E$ has an affine equation $y^2=x^3+4ax^2+2a^2x$. One checks that this equation is minimal and that $2$ is a prime of bad reduction. Running Tate's algorithm (see~\cite[IV.9]{SilvermanAdv}, for example) shows that if $p$ is a prime of bad reduction for $E$ then the reduction type is additive and the Tamagawa number $|E(\bbQ_p)/E_0(\bbQ_p)|$ is either $2$ or $4$, so it suffices to show that $E_0(\bbQ_p)$ is $3$-divisible. This follows from the description of $E_0(\bbQ_p)$ given in~\cite[Theorem~1]{Rene}. 
\end{proof}

\begin{remark}
The analogue of Lemma~\ref{lem:div-2} for elliptic curves with complex multiplication by $\bbZ[\frac{1+\sqrt{-11}}{2}]$ is false. For example, the elliptic curve $E/\bbQ$ with LMFDB label 121.b2, which has affine equation $y^2+y=x^3-x^2-7x+10$, 
has good (supersingular) reduction at $2$ and $E(\bbQ_2)/E_1(\bbQ_2)\cong \bbZ/3\bbZ$, so $E(\bbQ_2)$ is not $3$-divisible. This is why $\bbQ(\sqrt{-11})$ was excluded from Theorem~\ref{thm:const}; further investigation would be needed to determine whether the statement still holds in that case.
\end{remark}

\begin{lemma}\label{lem:div}
Let $p$ and $\ell$ be distinct primes with $\ell>3$. Let $a\in\bbQ_p^\times$, let $E^a$ denote the elliptic curve with affine equation $y^2=x^3+a$ and if $p$ is odd suppose that $E^a/\bbQ_p$ has bad reduction. Then $E^a(\bbQ_p)$ is $\ell$-divisible. 
\end{lemma}

\begin{proof}
First suppose that $E^a/\bbQ_p$ has bad reduction. An examination of Tate's algorithm (see~\cite[IV.9]{SilvermanAdv}, for example) shows that $E^a$ has additive reduction at $p$ and the Tamagawa number $[E^a(\bbQ_p):E^a_0(\bbQ_p)] $ is at most $4$. In particular, the Tamagawa number is coprime to $\ell$ and thus the claim is proved once we have shown that $E^a_0(\bbQ_p)$ is $\ell$-divisible. The $\ell$-divisibility of $E^a_0(\bbQ_p)$ follows from the description of this group given in~\cite[Theorem~1]{Rene}. 

Now suppose that $p=2$ and $E^a/\bbQ_2$ has good reduction. Tate's algorithm shows that this can only happen if $\ord_2(a)=4$ and $E^a$ has a minimal Weierstrass equation of the form $y^2+y=x^3+b$ for some $b\in\bbZ_2$. The standard filtration on the $\bbQ_2$ points of $E^a$ is
\[E^a(\bbQ_2)\supset E^a_1(\bbQ_2)\supset E^a_2(\bbQ_2)\supset \dots \]
where $E^a_1(\bbQ_2)$ denotes the kernel of the reduction map. The theory of formal groups (see~\cite[IV, VII]{Silverman}, for example) shows that $E^a_2(\bbQ_2)\cong 4\bbZ_2$, which is $\ell$-divisible, and  $E^a_1(\bbQ_2)/E^a_2(\bbQ_2)\cong \bbZ/2\bbZ$. Therefore, $E^a_1(\bbQ_2)$ is $\ell$-divisible. Finally, $E^a(\bbQ_2)/E^a_1(\bbQ_2)\cong \tilde{E}(\bbF_2)\cong\bbZ/3\bbZ$, whence it follows that $E^a(\bbQ_2)$ is $\ell$-divisible, as required.
\end{proof}

\begin{proof}[Proof of Theorem~\ref{thm:const}]
The statement for the infinite place is clear, since $\Br(\bbR)=\bbZ/2$ has trivial $\ell$-torsion. The statement for odd primes of good reduction for $Y$ follows from~\cite[Theorem~D, Remark~1.6.2]{BNwild}.

If $\cA\in\Br(\bbQ)$ then there is nothing to prove, so henceforth we will assume that $\cA\in\Br(Y)_\ell\setminus\Br(\bbQ)$. Thus, by Lemma~\ref{lem:varphigen}, $K$ is one of $\bbQ(\zeta_3), \bbQ(i),\bbQ(\sqrt{-2})$ and $\ell\leq 7$. 

For $K=\bbQ(i)$, we have $(\Br(Y)/\Br_1(Y))_{\textrm{odd}}\cong\bbZ/\ell\bbZ$ by~\eqref{ISgp}. For $K=\bbQ(\sqrt{-2})$, Theorem~\ref{thm:mainquad} shows that $\ell=3$ and $\Br(Y)/\Br_1(Y)\cong\bbZ/3\bbZ$. For $K=\bbQ(\zeta_3)$, we are assuming that $\ell>3$, and Theorem~\ref{thm:intro57} gives $\Br(Y)/\Br_1(Y)\cong\bbZ/\ell\bbZ$.
So from now on let $p\neq \ell$ be a finite prime and if $p$ is odd assume that $Y$ has bad reduction at $p$. Our task is to show that $\ev_{\cA,p}$ is constant.

Let $0\neq\varphi\in\Hom_{\Gamma_\bbQ}(E'_\ell, E_\ell)^-$ and let $\cB$ be the element of $\Br(Y)_\ell$ constructed from $\varphi$ as in Section~\ref{sec:prelim_eval}. By construction, and Lemma~\ref{lem:varphigen}, $\cB$ generates $\Br(Y)_\ell/\Br_1(Y)_\ell=\Br(Y)_\ell/\Br(\bbQ)_\ell$. Therefore, it suffices to prove that $\ev_\cB:Y(\bbQ_p)\to\Br(\bbQ_p)_\ell$ is the zero map.

Let $E$ and $E'$ have affine equations $E: y^2=f(x)$ and $E': y^2=g(x)$, respectively. Then $Y$ is the minimal desingularisation of the projective surface with affine equation
\begin{equation}\label{eq:affine}
u^2=f(x)g(t).
\end{equation}
Note that if $E$ and $E'$ are both replaced by their quadratic twists by some $\lambda\in\bbQ^\times$, with affine equations $E^\lambda: \lambda y^2=f(x)$ and $E'^\lambda:\lambda y^2=g(x)$, respectively,
then the resulting Kummer surface has affine equation
\[\lambda^2 u^2= f(x)g(t)\]
and the map $(x,t,u)\mapsto (x,t,\lambda u)$ gives an isomorphism back to the original model, showing that the Kummer surface remains the same. 

We adapt the arguments given in~\cite[Section~5]{I-S} to our setting. 
Since the evaluation map $\ev_{\cB,p}:Y(\bbQ_p)\to \Br(\bbQ_p)_\ell$ is locally constant, it is enough to show that it is zero on all $\bbQ_p$-points $R=(x_0,t_0,u_0)$ satisfying~\eqref{eq:affine}. Let $\delta_R=g(t_0)$. Again by local constancy, we are free to use the implicit function theorem to replace $R$ by a point $R'=(x_1,t_1,u_1)$ satisfying~\eqref{eq:affine}, sufficiently close to $R$, such that $\delta=\delta_{R'}\in\bbQ^\times$ and $u_1\neq 0$. Now $R'$ gives rise to points $P=(x_1, u_1/\delta)\in E^\delta(\bbQ_p)\setminus E^\delta_2$ and $Q=(t_1,1)\in E'^\delta(\bbQ_p)\setminus E'^\delta_2$. Note that $\Kum(E^\delta\times E'^\delta)$ is the minimal desingularisation of the projective surface with affine equation
$\delta^2 u^2=f(x)g(t)$ and the map $(x,t,u)\mapsto (x,t,\delta u)$ gives a $\bbQ$-isomorphism from $\Kum(E^\delta\times E'^\delta)$ to $Y$ that sends the point corresponding to $(P,Q)\in E^\delta(\bbQ_p)\times E'^\delta(\bbQ_p)$ to $R'$. Let $\varphi^\delta\in \Hom_{\Gamma_\bbQ}(E'^\delta_\ell, E^\delta_\ell)^-$ denote the isomorphism $E'^\delta_\ell\to E^\delta_\ell$ coming from $\varphi$. 
Now~\eqref{eq:pairingcup} shows that 
\begin{equation}\label{eq:Bpair}
\cB(R')=\chi_{P}\cup\varphi^\delta_*(\chi_{Q})\in\Br(\bbQ_p)_\ell\cong \ell^{-1}\bbZ/\bbZ
\end{equation}
where $\chi_P$ is the image of $P$ under $\chi: E^\delta(\bbQ_p)\to \mathrm{H}^1(\bbQ_p, E^\delta_\ell)$ and $\chi_Q$ is the image of $Q$ under $\chi: E'^\delta(\bbQ_p)\to \mathrm{H}^1(\bbQ_p, E'^\delta_\ell)$. The maps denoted by $\chi$ factor through the quotients $E^\delta(\bbQ_p)/\ell$ and $E'^\delta(\bbQ_p)/\ell$, respectively.

Recall that $p$ is either equal to $2$ or a prime of bad reduction for $Y$. By~\cite[Lemma~4.2]{goodred}, odd primes of good reduction for an abelian surface are primes of good reduction for the corresponding Kummer surface. Therefore, switching $E^\delta$ and $E'^\delta$ if necessary, if $p$ is odd then we may assume that $p$ is a prime of bad reduction for $E^\delta$.
Now Lemmas~\ref{lem:div-2},~\ref{lem:div} and~\cite[Section~5.2]{I-S} show that $E^\delta(\bbQ_p)/\ell=0$ and hence~\eqref{eq:Bpair} shows that $\cB(R')=0$, as required.
\end{proof}

\section{Surjective evaluation maps}\label{sec:pair_nonconst}

In this section, we prove Theorem~\ref{thm:mainWA}. The notation and assumptions of that theorem will be in force throughout this section. Furthermore, henceforth let $\varphi$ be a non-zero element of $\Hom_{\Gamma_\bbQ}(E'_\ell, E_\ell)^{-}$.

\begin{proposition}\label{prop:suffice}
To prove Theorem~\ref{thm:mainWA}, it suffices to prove the existence of $P\in E(\bbQ_\ell)$ and $Q\in E'(\bbQ_\ell)$ such that 
\[\chi_P\cup\varphi_*(\chi_Q)\neq 0\]
where $\chi_P$ is the image of $P$ under $\chi: E(\bbQ_\ell)\to \mathrm{H}^1(\bbQ_\ell, E_\ell)$ and $\chi_Q$ is the image of $Q$ under $\chi: E'(\bbQ_\ell)\to \mathrm{H}^1(\bbQ_\ell, E'_\ell)$.
\end{proposition}

\begin{proof}
Lemma~\ref{lem:varphigen} shows that $\varphi$ is an isomorphism. Let $\cB\in\Br(Y)_\ell$ be the element constructed from $\varphi$ as in Section~\ref{sec:prelim_eval}. By construction, and Lemma~\ref{lem:varphigen}, $\cB$ generates $\Br(Y)_\ell/\Br_1(Y)_\ell=\Br(Y)_\ell/\Br(\bbQ)_\ell$. Therefore, it suffices to prove Theorem~\ref{thm:mainWA} with $\cA=\cB$. By continuity, we may assume that $P$ and $Q$ are not $2$-torsion points (replacing them with nearby points if necessary).
Let $[P,Q]\in Y(\bbQ_\ell)$ denote the point of $Y$ coming from $(P,Q)$. Then~\eqref{eq:pairingcup} shows that
\[0\neq\chi_P\cup\varphi_*(\chi_Q)=\cB([P,Q])\in\Br(\bbQ_\ell)_\ell\cong \ell^{-1}\bbZ/\bbZ.\]
Furthermore, for all $m\in\bbZ$ such that $mP\notin E_2$, we have 
\[\cB([mP,Q])=\chi_{mP}\cup\varphi_*(\chi_Q)=m(\chi_P\cup\varphi_*(\chi_Q)).\]
Therefore, taking scalar multiples of $P$ gives the desired surjectivity. (Again, we can always substitute a sufficiently close point to avoid any issues with $2$-torsion points.)

Finally, it is well known and easy to verify that having a non-constant evaluation map $\ev_{\cB,\ell}: Y(\bbQ_\ell)\to \Br(\bbQ_\ell)_\ell$ implies the existence of an adelic point on $Y$ that does not pair to zero with $\cB$ under the Brauer--Manin pairing~\eqref{eq:BM}. 
\end{proof}

Therefore, our task is to find $P\in E(\bbQ_\ell)$ and $Q\in E'(\bbQ_\ell)$ such that 
$\chi_P\cup\varphi_*(\chi_Q)\neq 0$, as in Proposition~\ref{prop:suffice}. We begin by treating the easier case where the order $\ell$ of our Brauer group element splits in the CM field $K$.

\subsection{The case of $\ell$ split in $K/\bbQ$}\label{sec:split}
This section mimics~\cite[\S 5.3]{I-S}. Let $\ell$ be an odd prime number that splits in $K/\bbQ$, so $\ell=\lambda\bar{\lambda}$ for some $\lambda\in \cO_K$. This splitting will allow us to replace the cup-product pairing that gives the evaluation map (see~\eqref{eq:pairingcup}) with a non-degenerate pairing $\mathrm{H}^1(\bbQ_\ell, E_{\lambda})\times \mathrm{H}^1(\bbQ_\ell, E'_{\lambda})\rightarrow \Br(\bbQ_\ell)_\ell$ (see~\eqref{eq:pairing2} below). The proof of Theorem~\ref{thm:mainWA} in this setting will then come down to showing that the images of $E(\bbQ_\ell)$ and $E'(\bbQ_\ell)$ in $\mathrm{H}^1(\bbQ_\ell, E_{\lambda})$ and $\mathrm{H}^1(\bbQ_\ell, E'_{\lambda})$, respectively, are sufficiently large.

Choose an embedding of $K$ into $\bbQ_\ell$ such that $\lambda$ is a uniformiser of $\bbZ_\ell$ and $\bar{\lambda}\in\bbZ_\ell^\times$.
Now $E_\ell=E_\lambda\oplus E_{\bar{\lambda}}$ as $\Gamma_{\bbQ_\ell}$-modules and therefore $\mathrm{H}^1(\bbQ_\ell, E_\ell)=\mathrm{H}^1(\bbQ_\ell, E_\lambda)\oplus \mathrm{H}^1(\bbQ_\ell, E_{\bar{\lambda}})$. Since the restriction of the skew-symmetric Weil pairing to each of the one-dimensional $\bbF_\ell$-subspaces $E_\lambda$ and $E_{\bar{\lambda}}$
is trivial, each of the subspaces $\mathrm{H}^1(\bbQ_\ell, E_\lambda)$ and $\mathrm{H}^1(\bbQ_\ell, E_{\bar{\lambda}})$ is isotropic for the pairing
\begin{equation*}
\label{pairing0}
\cup: \mathrm{H}^1(\bbQ_\ell, E_\ell)\times \mathrm{H}^1(\bbQ_\ell,E_\ell)\rightarrow \Br(\bbQ_\ell)_\ell\cong\ell^{-1}\bbZ/\bbZ
\end{equation*}
described in~\eqref{eq:cup}. By the non-degeneracy of the cup product,
these subspaces are maximal isotropic subspaces of $\mathrm{H}^1(\bbQ_\ell,E_\ell)$, each of dimension
$\frac{1}{2} \dim\mathrm{H}^1(\bbQ_\ell,E_\ell)$. Thus,~\eqref{eq:cup} induces a non-degenerate pairing
\begin{equation}
\label{pairing1}
\cup: \mathrm{H}^1(\bbQ_\ell, E_\lambda)\times\mathrm{H}^1(\bbQ_\ell, E_{\bar{\lambda}})\rightarrow \Br(\bbQ_\ell)_\ell\cong\ell^{-1}\bbZ/\bbZ.
\end{equation}

Recall that we have
\[0\neq\varphi\in\Hom_{\Gamma_\bbQ}(E'_\ell, E_\ell)^{-}=\{\psi\in\Hom_{\Gamma_\bbQ}(E'_\ell, E_\ell)\mid \psi \gamma=\bar{\gamma}\psi\ \forall\ \gamma\in\cO_K\}.\]
Then $\varphi=\varphi'+\varphi''$ where $\varphi'$ is an isomorphism of $\Gamma_{\bbQ_\ell}$-modules $E'_{\lambda}\rightarrow E_{\bar{\lambda}}$ and $\varphi''$ is an isomorphism of $\Gamma_{\bbQ_\ell}$-modules $E'_{\bar{\lambda}}\rightarrow E_{\lambda}$. 
Consequently, the induced isomorphism $\varphi_*:\mathrm{H}^1(\bbQ_\ell,E'_\ell)\rightarrow \mathrm{H}^1(\bbQ_\ell, E_\ell)$ is a sum of $\varphi'_*: \mathrm{H}^1(\bbQ_\ell, E'_{\lambda})\rightarrow \mathrm{H}^1(\bbQ_\ell,E_{\bar{\lambda}})$ and $\varphi''_*: \mathrm{H}^1(\bbQ_\ell, E'_{\bar{\lambda}})\rightarrow \mathrm{H}^1(\bbQ_\ell,E_{\lambda})$. In conclusion, \eqref{pairing1} together with $\varphi'_*$ induces a non-degenerate pairing
\begin{equation}
\label{eq:pairing2}
\mathrm{H}^1(\bbQ_\ell, E_{\lambda})\times \mathrm{H}^1(\bbQ_\ell, E'_{\lambda})\rightarrow \Br(\bbQ_\ell)_\ell\cong \ell^{-1}\bbZ/\bbZ.
\end{equation}

We will use the non-degeneracy of the pairing~\eqref{eq:pairing2} to prove Theorem~\ref{thm:mainWA} via Proposition~\ref{prop:suffice} in the case where $\ell$ splits in $K/\bbQ$. For this to work, we will need to show that the images of $E(\bbQ_\ell)$ and $E'(\bbQ_\ell)$ in $\mathrm{H}^1(\bbQ_\ell, E_{\lambda})$ and $\mathrm{H}^1(\bbQ_\ell, E'_{\lambda})$, respectively, are sufficiently large. 
We will frequently use that $\mathrm{H}^1(\bbQ_\ell, E_{\lambda})$ and $E(\bbQ_\ell)/\ell$ are both maximal isotropic subspaces of $\mathrm{H}^1(\bbQ_\ell, E_\ell)$. Moreover,
\[E(\bbQ_\ell)/\lambda\subset \mathrm{H}^1(\bbQ_\ell, E_{\lambda})\]
and $E(\bbQ_\ell)/\ell=E(\bbQ_\ell)/\lambda\oplus E(\bbQ_\ell)/\bar{\lambda}$.

We need to treat four cases as delineated in Lemma~\ref{lem:varphigen}: the CM field $K$ is one of $\bbQ(\zeta_3),\bbQ(i), \bbQ(\sqrt{-2}), \bbQ(\sqrt{-11})$. 

We begin with the case $K=\bbQ(\sqrt{-2})$. In this case, Theorem~\ref{thm:mainquad} shows that $\ell=3$ and $Y=\Kum(E\times E^a)$ where $E/\bbQ$ is the elliptic curve with affine equation $y^2=x^3+4x^2+2x$, $a\in\{-3,6\}$ and $E^a$ denotes the quadratic twist of $E$ by $a$.

\begin{proposition}\label{prop:maxa}
Let $E/\bbQ$ have affine equation $y^2=x^3+4x^2+2x$ and let $E^a$ denote its quadratic twist by $a\in\bbQ^\times$. Choose an embedding of $\bbQ(\sqrt{-2})$ in $\bbQ_3$ such that $\sqrt{-2}\equiv 1\bmod{3}$, whereby $\lambda=1-\sqrt{-2}$ is a uniformiser for $\bbZ_3$ and $\bar{\lambda}=1+\sqrt{-2}\in\bbZ_3^\times$. Then  
 \[E(\bbQ_3)/\lambda\cong E(\bbQ_3)/\bar{\lambda}\cong\bbZ/3\bbZ\]
and for $a\in\{-3,6\}$ we have $E^a(\bbQ_3)/\bar\lambda=0$ and
 \[\mathrm{H}^1(\bbQ_3, E^a_{\lambda})=E^a(\bbQ_3)/\lambda=E^a(\bbQ_3)/3\cong (\bbZ/3\bbZ)^2. \]
\end{proposition}

\begin{proof}
First, we prove the statement for $E$. This elliptic curve has good reduction at $3$ and we have 
\[E(\bbQ_3)/E_1(\bbQ_3)\cong \tilde{E}(\bbF_3)=\{\cO, (0,0), (1,1), (1,-1), (-1,1), (-1,-1)\}.\]
Let $\hat{E}$ denote the formal group associated to $E$. We have isomorphisms of topological groups
\begin{equation}\label{eq:E1log}
E_1(\bbQ_3)\xrightarrow{(x,y)\mapsto \frac{-x}{y}} \hat{E}(3\bbZ_3)\xrightarrow{\log} 3\bbZ_3
\end{equation}
where $\log$ denotes the formal logarithm, which is given by a power series $T+\sum_{n\geq 2}\frac{b_n}{n}T^n$ for some $b_n\in\bbZ$, see~\cite[Proposition~IV.5.5]{Silverman}. These isomorphisms respect the action of $\sqrt{-2}$.
This means that $\bar{\lambda}=1+\sqrt{-2}$ acts as a multiplication by a unit in $\bbZ_3$ and $\lambda=1-\sqrt{-2}$ acts as multiplication by a uniformiser and hence 
\begin{equation}\label{E-2,1}
E_1(\bbQ_3)/\lambda=E_1(\bbQ_3)/3\cong\bbZ/3.
\end{equation}
Now one checks that $\lambda$ induces an automorphism of $E(\bbQ_3)/E_1(\bbQ_3)\cong\tilde{E}(\bbF_3)$. Thus, 
\begin{equation}\label{E-2,3}
E(\bbQ_3)=\lambda E(\bbQ_3)+E_1(\bbQ_3)
\end{equation} and 
\[E(\bbQ_3)/\lambda=E_1(\bbQ_3)/\lambda\cong \bbZ/3\]
by~\eqref{E-2,1}. Furthermore, \eqref{E-2,3} gives 
\[\bar\lambda E(\bbQ_3)=3E(\bbQ_3)+\bar\lambda E_1(\bbQ_3)=3E(\bbQ_3)+E_1(\bbQ_3)\]
since $\bar{\lambda}$ acts as an automorphism on $E_1(\bbQ_3)$ by our discussion above. Therefore,
\[E(\bbQ_3)/\bar\lambda=E(\bbQ_3)/(3E(\bbQ_3)+E_1(\bbQ_3))\cong \tilde{E}(\bbF_3)/3\cong \bbZ/3.\]

Now we deal with $E^{a}$ for $a\in\{-3,6\}$. In fact, since $-2\in\bbQ_3^{\times 2}$, $E^{-3}$ and $E^6$ are isomorphic over $\bbQ_3$ and we may take $a=-3$ in what follows.
The elliptic curve $E^{-3}$ has additive reduction at $3$ and the Tamagawa number $|E^{-3}(\bbQ_3)/E^{-3}_0(\bbQ_3)|$ is $2$, so we can replace $E^{-3}(\bbQ_3)$ by $E^{-3}_0(\bbQ_3)$ in our calculations. 
As above, we find that
\begin{equation}\label{E-2,a}
E^{-3}_1(\bbQ_3)/\lambda=E^{-3}_1(\bbQ_3)/3\cong\bbZ/3
\end{equation}
and 
\begin{equation}\label{E-2,a0}
E^{-3}_1(\bbQ_3)/\bar\lambda=0.
\end{equation}
Furthermore, $E^{-3}_0(\bbQ_3)/E^{-3}_1(\bbQ_3)\cong \tilde{E}^{-3}_{\text {ns}}(\bbF_3)=\{\cO,(1,1), (1,-1)\}$ and $\sqrt{-2}$ acts as the identity on $\tilde{E}^{-3}_{\text {ns}}(\bbF_3)$, whereby
\begin{equation}\label{E-2,asum}
E^{-3}_0(\bbQ_3)=\bar\lambda E^{-3}_0(\bbQ_3)+E^{-3}_1(\bbQ_3)
\end{equation}
and hence
\[E^{-3}_0(\bbQ_3)/\bar\lambda=E^{-3}_1(\bbQ_3)/\bar\lambda=0\]
by~\eqref{E-2,a0}. Moreover, \eqref{E-2,asum} and \eqref{E-2,a} give
\[\lambda E^{-3}_0(\bbQ_3)=3E^{-3}_0(\bbQ_3)+\lambda E^{-3}_1(\bbQ_3)=3E^{-3}_0(\bbQ_3).\]
Now~\cite[Theorem~1]{Rene} shows that $E^{-3}_0(\bbQ_3)\cong 3\bbZ_3\times \bbZ/3\bbZ$ and hence
\[E^{-3}_0(\bbQ_3)/\lambda=E^{-3}_0(\bbQ_3)/3\cong(\bbZ/3\bbZ)^2. \]
Now since $E^{-3}(\bbQ_3)/3=E^{-3}(\bbQ_3)/\lambda\subset \mathrm{H}^1(\bbQ_3, E^{-3}_{\lambda})$, and $E^{-3}(\bbQ_3)/3$ and $\mathrm{H}^1(\bbQ_3, E^{-3}_{\lambda})$ are both maximal isotropic subspaces of $\mathrm{H}^1(\bbQ_3, E^{-3}_3)$, they must be equal.
\end{proof}

\begin{corollary}\label{cor:surj-2}
Let $E/\bbQ$ have affine equation $y^2=x^3+4x^2+2x$ and let $E^a$ denote its quadratic twist by $a\in\bbQ^\times$. Then for $a\in\{-3,6\}$, there exist $P\in E(\bbQ_3)$ and $Q\in E^a(\bbQ_3)$ such that $\chi_P\cup\varphi_*(\chi_Q)$ is non-zero.
\end{corollary}

\begin{proof}
This follows from the non-degeneracy of~\eqref{eq:pairing2}, with $\ell=3$ and $\lambda=1-\sqrt{-2}$, and Proposition~\ref{prop:maxa}.
\end{proof}

Next, we treat the case where $K=\bbQ(\sqrt{-11})$. In this case, Theorem~\ref{thm:mainquad} shows that $\ell=3$ and $Y=\Kum(E\times E^a)$ where $E/\bbQ$ has affine equation $y^2=x^3-2^5\cdot3^3\cdot11x+2^4\cdot3^3\cdot7\cdot 11^2$, $a\in\{-3,33\}$ and $E^a$ denotes the quadratic twist of $E$ by $a$. A minimal Weierstrass equation for $E$ is given by $y^2+y=x^3-x^2-7x+10$, see~\cite{LMFDB}.

\begin{proposition}\label{prop:maxa-11}
Let $E/\bbQ$ have affine equation $y^2+y=x^3-x^2-7x+10$ and let $E^a$ denote its quadratic twist by $a\in\bbQ^\times$. Choose an embedding of $\bbQ(\sqrt{-11})$ in $\bbQ_3$ such that $\sqrt{-11}\equiv 1\bmod{3}$, whereby $\lambda=(1-\sqrt{-11})/2$ is a uniformiser for $\bbZ_3$ and $\bar{\lambda}=(1+\sqrt{-11})/2\in\bbZ_3^\times$. Then for $a\in \{1,-3,33\}$ we have $E^a(\bbQ_3)/\bar{\lambda}=0$ and
 \[\mathrm{H}^1(\bbQ_3, E^a_{\lambda})=E^a(\bbQ_3)/\lambda=E^a(\bbQ_3)/3 \cong \bbZ/3\bbZ.\]
\end{proposition}

\begin{proof}
Standard calculations similar to those in the proof of Proposition~\ref{prop:maxa} show that 
\begin{equation*}
E(\bbQ_3)/\bar{\lambda} =0 \ \textrm{ and }\  E(\bbQ_3)/\lambda=E(\bbQ_3)/3\cong\bbZ/3.
\end{equation*}
Now since $E(\bbQ_3)/3=E(\bbQ_3)/\lambda\subset \mathrm{H}^1(\bbQ_3, E_{\lambda})$ and both are maximal isotropic subspaces of $\mathrm{H}^1(\bbQ_3, E_3)$, they must be equal. This completes the proof of the statement for $E$.

Now we deal with $E^{a}$ for $a\in\{-3,33\}$. Since $-11\in \bbQ_3^{\times 2}$, it suffices to take $a=-3$.
Using~\cite[\S4.3]{Connell} and Tate's algorithm, we find that $E^{-3}(\bbQ_3)=E^{-3}_0(\bbQ_3)$. 
Moreover, the isomorphism~\eqref{eq:E1log} respects complex multiplication. Therefore, 
\begin{equation*}
E_1^{-3}(\bbQ_3)/\bar{\lambda} =0 \ \textrm{ and }\  E_1^{-3}(\bbQ_3)/\lambda=E_1^{-3}(\bbQ_3)/3\cong\bbZ/3.
\end{equation*}
By~\cite[Definition~10 and Proposition~11]{Rene}, the map $(x,y)\mapsto -x/y$ extends to an isomorphism of topological groups $E^{-3}_0(\bbQ_3)\rightarrow \hat{E}^{-3}(\bbZ_3)$. 
Furthermore,~\cite[Proposition~18]{Rene} gives an isomorphism of topological groups $\hat{E}^{-3}(\bbZ_3)\rightarrow \bbZ_3$ extending the isomorphism $ \hat{E}^{-3}(3\bbZ_3)\rightarrow 3\bbZ_3$ given by the formal logarithm. Transporting the action of complex multiplication by $\lambda$ along these isomorphisms gives an endomorphism of $\bbZ_3$ that coincides with multiplication by $\lambda$ on $3\bbZ_3$. But this endomorphism must then be multiplication by $\lambda$. The same argument applies to $\bar{\lambda}$ and hence we have 
\begin{equation}\label{E-11,1}
E_0^{-3}(\bbQ_3)/\bar{\lambda} =0 \ \textrm{ and }\  E_0^{-3}(\bbQ_3)/\lambda=E_0^{-3}(\bbQ_3)/3\cong\bbZ/3.
\end{equation}
The usual argument about maximal isotropic subspaces of $\mathrm{H}^1(\bbQ_3, E^{-3}_3)$ completes the proof.
\end{proof}

\begin{corollary}\label{cor:surj-11}
Let $E/\bbQ$ have affine equation $y^2=x^3-x^2-7x+10$ and let $E^a$ denote its quadratic twist by $a\in\bbQ^\times$. Then for $a\in\{-3,33\}$, there exist $P\in E(\bbQ_3)$ and $Q\in E^a(\bbQ_3)$ such that $\chi_P\cup\varphi_*(\chi_Q)$ is non-zero.
\end{corollary}

\begin{proof}
This follows from the non-degeneracy of~\eqref{eq:pairing2}, with $\ell=3$ and $\lambda=(1-\sqrt{-11})/2$, and Proposition~\ref{prop:maxa-11}.
\end{proof}

Now we treat the case where $K=\bbQ(\zeta_3)$. By Lemma~\ref{lem:varphigen}, we have $\ell\leq 7$. In this section, our focus is on the case where $\ell$ splits in $K/\bbQ$, so we take $\ell=7$. For $a\in\bbQ^\times$, let $E^a$ denote the ellliptic curve over $\bbQ$ with affine equation $y^2=x^3+a$.
Now Theorems~\ref{thm:intro57} and~\ref{thm:Br1} give $Y=\Kum(E^c\times E^d)$ where $c,d\in\bbQ^\times$ satisfy $-2^4\cdot 7^{-1}\cdot cd\in\langle -3^3\rangle\bbQ^{\times 6}$.

\begin{proposition}
\label{prop:subspace}
For any $a\in\bbQ_7^\times$, let $E$ be the elliptic curve with affine equation $y^2=x^3+a$. Choose an embedding of $\bbQ(\zeta_3)$ in $\bbQ_3$ such that $\zeta_3 \equiv 2\bmod{7}$, whereby $\lambda=1-2\zeta_3^2=3+2\zeta_3$ is a uniformiser for $\bbZ_7$ and $\bar{\lambda}=1-2\zeta_3\in\bbZ_7^\times$. 
\begin{enumerate}
\item \label{item14} If $a\in 2\cdot 7\cdot \bbQ_7^{\times 6}$ then $E^a(\bbQ_7)/\bar{\lambda}=0$ and
\[\mathrm{H}^1(\bbQ_7, E^a_{\lambda})=E^a(\bbQ_7)/\lambda=E^a(\bbQ_7)/7\cong (\bbZ/7\bbZ)^2. \]
\item\label{item-2} If $a\in -2\cdot \bbQ_7^{\times 6}$ then $E^a(\bbQ_7)/\lambda\cong E^a(\bbQ_7)/\bar{\lambda}\cong \bbZ/7\bbZ.$
\item\label{itemelse} In all other cases, $E^a(\bbQ_7)/\bar{\lambda}=0$ and
\[\mathrm{H}^1(\bbQ_7, E^a_{\lambda})=E^a(\bbQ_7)/\lambda=E^a(\bbQ_7)/7\cong \bbZ/7\bbZ. \]
\end{enumerate}
\end{proposition}

\begin{proof}
Recall that $E^a(\bbQ_7)/7$ and $\mathrm{H}^1(\bbQ_7, E^a_{\lambda})$ are both maximal isotropic subspaces of $\mathrm{H}^1(\bbQ_7,E^a_7)$, and $E^a(\bbQ_7)/\lambda\subset \mathrm{H}^1(\bbQ_7,E^a_{\lambda})$. Therefore, in cases where we show that $E^a(\bbQ_7)/7=E^a(\bbQ_7)/\lambda$, it will follow that this group is also equal to $\mathrm{H}^1(\bbQ_7, E^a_{\lambda})$.

Since the $\bbQ_7$-isomorphism class of $E^a$ only depends on the class of $a$ in $\bbQ_7^\times/\bbQ_7^{\times 6}$, we may assume that $0\leq\ord_7(a)\leq 5$. The reduction type of $E^a$ is either good or additive. 
The Tamagawa number $|E^a(\bbQ_7)/E^a_0(\bbQ_7)|$ is 
coprime to $7$, so we can replace $E^a(\bbQ_7)$ by $E^a_0(\bbQ_7)$ in our calculations. Standard calculations give 
\begin{align}\label{eq:E1-7}
\bar{\lambda}E^a_1(\bbQ_7)=E^a_1(\bbQ_7) \ \textrm{  and  }\ \lambda E^a_1(\bbQ_7)=7E^a_1(\bbQ_7).
\end{align}
If $a\equiv 14\pmod{49}$, then \cite[Proposition~18]{Rene} shows that the extension $0\rightarrow E^a_1(\bbQ_7)\rightarrow E^a_0(\bbQ_7)\rightarrow \tilde{E}^a_{\text{ns}}(\bbF_7)\rightarrow 0$ is split. Computing the action of $\zeta_3$ on $\tilde{E}^a_{\text{ns}}(\bbF_7)$ shows that it coincides with multiplication by $2$. Therefore, $\bar{\lambda}\tilde{E}^a_{\text{ns}}(\bbF_7)=\tilde{E}^a_{\text{ns}}(\bbF_7)$ and $\lambda\tilde{E}^a_{\text{ns}}(\bbF_7)=7\tilde{E}^a_{\text{ns}}(\bbF_7)=0$. When combined with ~\eqref{eq:E1-7}, this proves part~\eqref{item14} of the proposition.

In the case where $E^a$ has additive reduction and $a\not\equiv 14\pmod{49}$, we use~\cite[Definition~10, Proposition~11 and Proposition~18]{Rene} to obtain an isomorphism $E^a_0(\bbQ_7)\cong\bbZ_7$ respecting the action of $\zeta_3$, and hence show that $\bar{\lambda}E^a_0(\bbQ_7)=E^a_0(\bbQ_7)$ and $\lambda E^a_0(\bbQ_7)=7E^a_0(\bbQ_7)$, proving the proposition in this case.






Now suppose that $E^a$ has good reduction so $E^a(\bbQ_7)/E^a_1(\bbQ_7)\cong \tilde{E}^a(\bbF_7)$. Elementary calculations show that $|\tilde{E}^a(\bbF_7)|$ is coprime to $7$, unless $a\in -2\cdot\bbQ_7^{\times 6}$ when $\tilde{E}^a(\bbF_7)=\tilde{E}^{-2}(\bbF_7)\cong \bbZ/7\bbZ$. Thus,~\eqref{eq:E1-7} proves the proposition
 for $a\notin -2\cdot\bbQ_7^{\times 6}$.
 
Our final task is to prove part~\eqref{item-2}. Taking $a=-2$, we have
\[\tilde{E}^{-2}(\bbF_7)=\{\cO, (3,2), (3,-2),(-2,2),(-2,-2),(-1,2), (-1,-2)\}\]
and multiplication by $\zeta_3$ sends $(x,y)$ to $(2x,y)$, which coincides with multiplication by $4$ on $\tilde{E}^{-2}(\bbF_7)\cong\bbZ/7\bbZ$. Therefore, $\bar{\lambda}\tilde{E}^{-2}(\bbF_7)=0$ and hence $\bar{\lambda}E^{-2}(\bbQ_7)\subset E^{-2}_1(\bbQ_7)$. By~\eqref{eq:E1-7}, this containment is an equality and hence 
\[E^{-2}(\bbQ_7)/\bar{\lambda}=E^{-2}(\bbQ_7)/E^{-2}_1(\bbQ_7)\cong\tilde{E}^{-2}(\bbF_7)\cong \bbZ/7\bbZ.\]
Moreover, $\lambda \tilde{E}^{-2}(\bbF_7)=\tilde{E}^{-2}(\bbF_7)$ and hence 
\[E^{-2}(\bbQ_7)=\lambda E^{-2}(\bbQ_7)+E^{-2}_1(\bbQ_7),\]
whereby~\eqref{eq:E1-7} gives
\[E^{-2}(\bbQ_7)/\lambda=E^{-2}_1(\bbQ_7)/\lambda\cong \bbZ/7\bbZ. \qedhere\] 
\end{proof}

\begin{corollary}\label{cor:surj7}
For $a\in\bbQ^\times$, let $E^a$ denote the ellliptic curve over $\bbQ$ with affine equation $y^2=x^3+a$. Let $c,d\in\bbQ^\times$ satisfy $-2^4\cdot 7^{-1}\cdot cd\in\langle -3^3\rangle\bbQ^{\times 6}$. Then there exist $P\in E^c(\bbQ_7)$ and $Q\in E^d(\bbQ_7)$ such that $\chi_P\cup\varphi_*(\chi_Q)$ is non-zero.
\end{corollary}

\begin{proof}
First note that the relation $-2^4\cdot 7^{-1}\cdot cd\in\langle -3^3\rangle\bbQ^{\times 6}$ implies that $c\in 2\cdot 7\cdot \bbQ_7^{\times 6}$ if and only if $d\in -2 \cdot \bbQ_7^{\times 6}$.
Suppose that $c\in 2\cdot 7\cdot \bbQ_7^{\times 6}$. Then $d\in -2 \cdot \bbQ_7^{\times 6}$ and the result follows from the non-degeneracy of~\eqref{eq:pairing2}, with $\ell=7$ and $\lambda=1-2\zeta_3^2$, and Proposition~\ref{prop:subspace}\eqref{item14} and~\eqref{item-2}. The proof in the case where $c\in -2 \cdot \bbQ_7^{\times 6}$ follows by symmetry in $c$ and $d$.
In the remaining case, where $c$ and $d$ are in neither $2\cdot 7\cdot \bbQ_7^{\times 6}$ nor $-2 \cdot \bbQ_7^{\times 6}$, the result follows from the non-degeneracy of~\eqref{eq:pairing2} and Proposition~\ref{prop:subspace}\eqref{itemelse}.
\end{proof}

The remaining case is where $K=\bbQ(i)$ and~\eqref{ISgp} together with the splitting condition gives $\ell=5$ and $Y=\Kum(E^{m_1}\times E^{m_2})$ where $m_1,m_2\in\bbQ^\times$ satisfy $5^3 m_1 m_2\in \langle -4\rangle\bbQ^{\times 4}$ and $E^m/\bbQ$ is the elliptic curve with affine equation $y^2=x^3-mx$. This case was tackled by Ieronymou and Skorobogatov in~\cite{I-S}, en route to their treatment of diagonal quartic surfaces. In particular,~\cite[Proposition~5.7 and Corollary~5.8]{I-S} shows that if $m_1$ and $m_2$ are not in $(1\pm 2i)\bbQ_5^{\times 4}$ then there exist $P\in E^{m_1}(\bbQ_5)$ and $Q\in E^{m_2}(\bbQ_5)$ such that $\chi_P\cup\varphi_*(\chi_Q)$ is non-zero. If $m_1$ or $m_2$ is in $(1\pm 2i)\bbQ_5^{\times 4}$ then replace $m_1$ and $m_2$ by $5^2m_1$ and $5^2m_2$, respectively. Now the relation $5^3 m_1 m_2\in \langle -4\rangle\bbQ^{\times 4}$ implies that $5^2m_1$ and $5^2m_2$ are not in $(1\pm 2i)\bbQ_5^{\times 4}$ and one can apply~\cite[Proposition~5.7 and Corollary~5.8]{I-S} as before. The Kummer surface $Y$ is unchanged when we replace $m_1$ and $m_2$ by $5^2m_1$ and $5^2m_2$, respectively, because this simply amounts to replacing each of $E^{m_1}$ and $E^{m_2}$ by their quadratic twists by $5\in\bbQ^\times/\bbQ^{\times 2}$. We have seen previously (in the proof of Theorem~\ref{thm:const}) that this does not change the Kummer surface.

At this stage, thanks to Proposition~\ref{prop:suffice}, Corollaries~\ref{cor:surj-2},~\ref{cor:surj-11},~\ref{cor:surj7} and the above discussion for $K=\bbQ(i)$ and $\ell=5$, we have proved Theorem~\ref{thm:mainWA} in all cases where $\ell$ splits in $K$.

\subsection{The case of $\ell$ inert in $K/\bbQ$}\label{sec:inert}
This is the hardest case. We will follow the work of Ieronymou and Skorobogatov in~\cite{IScorrigendum} and use a result of Harpaz and Skorobogatov (Proposition~\ref{prop:HS} below) to reduce the proof of Theorem~\ref{thm:mainWA} in this case to the task of showing that the function fields of torsors associated to certain elements of $\mathrm{H}^1(\bbQ_5, E_{5})$ are not isomorphic. We will avoid difficult calculations with these totally wildly ramified extensions of degree $25$ by using quadratic twists to obtain function fields with distinct discriminants.

 By Lemma~\ref{lem:varphigen} and our assumptions in the statement of Theorem~\ref{thm:mainWA} that $\ell$ is odd, and if $K=\bbQ(\zeta_3)$ then $\ell>3$, we have excluded all cases where $\ell$ ramifies in $K/\bbQ$. For $K=\bbQ(\zeta_3)$, the evaluations at $3$-adic points of Brauer group elements of orders $3$ and $9$  will be studied in future work.

By Lemma~\ref{lem:varphigen} and Theorems~\ref{thm:mainquad},~\ref{thm:intro57} and~\eqref{ISgp}, the only cases where the order $\ell$ of our Brauer group element is inert in the CM field $K$ are when $\ell=5$ and $K=\bbQ(\zeta_3)$, and when $\ell=3$ and $K=\bbQ(i)$.

The case where $K=\bbQ(i)$ and $\ell=3$ was tackled by Ieronymou and Skorobogatov in~\cite{IScorrigendum}. In this case,~\eqref{ISgp} gives $Y=\Kum(E^{m_1}\times E^{m_2})$ where $m_1,m_2\in\bbQ^\times$ satisfy $-3 m_1 m_2\in \langle -4\rangle\bbQ^{\times 4}$ and $E^m/\bbQ$ is the elliptic curve with affine equation $y^2=x^3-mx$. By~\cite[Proposition~2.2]{IScorrigendum}, if the $3$-adic valuations of $m_1$ and $m_2$ are both non-zero modulo $4$, then there exist $P\in E^{m_1}(\bbQ_3)$ and $Q\in E^{m_2}(\bbQ_3)$ such that $\chi_P\cup\varphi_*(\chi_Q)\neq 0$. If the $3$-adic valuation of $m_1$ or $m_2$ is zero modulo $4$, then we can replace $m_1$ and $m_2$ by $3^2m_1$ and $3^2m_2$, respectively, which does not change the Kummer surface $Y$, and then apply~\cite[Proposition~2.2]{IScorrigendum}. Now applying Proposition~\ref{prop:suffice} completes the proof of Theorem~\ref{thm:mainWA} in this case. 

Note the qualitative difference in behaviour between the Kummer surfaces $Y$ and the closely related diagonal quartic surfaces studied by Ieronymou and Skorobogatov. In~\cite[Theorem~2.3]{IScorrigendum}, the authors show that for some diagonal quartic surfaces, a Brauer group element of order $3$ has constant evaluation on $3$-adic points, while for others it attains all three possible values.

The remainder of this section is devoted to proving the last remaining case of Theorem~\ref{thm:mainWA}, wherein $K=\bbQ(\zeta_3)$ and $\ell=5$. Henceforth, for $a\in\bbQ^\times$, we use $E^a$ to denote the elliptic curve over $\bbQ$ with affine equation $y^2=x^3+a$. Theorems~\ref{thm:intro57} and~\ref{thm:Br1} give $Y=\Kum(E^c\times E^d)$ where $c,d\in\bbQ^\times$ satisfy $2^4\cdot 5\cdot cd\in\langle -3^3\rangle\bbQ^{\times 6}$.
The work below should be compared to~\cite{IScorrigendum}.

We begin by showing that the image of $E^a(\bbQ_5)$ in $\mathrm{H}^1(\bbQ_5, E^a_{5})$ is a one-dimensional vector space over $\bbF_5$.

\begin{lemma}\label{lem:1dim}
Let $a\in\bbQ_5^\times$. Then $E^a(\bbQ_5)/5\cong \bbZ/5\bbZ$ as abelian groups.
\end{lemma}

\begin{proof}
An inspection of Tate's algorithm shows that $E^a/\bbQ_5$ has either good or additive reduction and in all cases the Tamagawa number $|E^a(\bbQ_5)/E^a_0(\bbQ_5)|$ is coprime to $5$. So it is enough to show that $E^a_0(\bbQ_5)/5\cong \bbZ/5$. In the case of additive reduction, this follows from~\cite[Theorem~1]{Rene}. Now suppose that $E^a/\bbQ_5$ has good reduction. Then we compute $E^a(\bbQ_5)/E^a_1(\bbQ_5)\cong \tilde{E^a}(\bbF_5)\cong \bbZ/6\bbZ$, so it suffices to show that $E^a_1(\bbQ_5)/5\cong \bbZ/5\bbZ$. The theory of formal groups gives $E^a_1(\bbQ_5)\cong 5\bbZ_5$.
\end{proof}

\begin{proposition}\label{prop:suffice5}
Let $c,d\in\bbQ^\times$ be such that $2^4\cdot 5\cdot cd\in\langle -3^3\rangle\bbQ^{\times 6}$.
Let $P$ and $Q$ generate $E^c(\bbQ_5)/5$ and $E^d(\bbQ_5)/5$, and let $\chi_P$ and $\chi_Q$ denote their images in $\mathrm{H}^1(\bbQ_5, E^c_{5})$ and $\mathrm{H}^1(\bbQ_5, E^d_{5})$, respectively.
To prove Theorem~\ref{thm:mainWA} for $Y=\Kum(E^c\times E^d)$ and $\ell=5$, it suffices to show that $\varphi_*(\chi_Q)$ is not a scalar multiple of $\chi_P$.
\end{proposition}

\begin{proof}
This follows from Proposition~\ref{prop:suffice}, Lemma~\ref{lem:1dim} and the fact that the image of $E^a(\bbQ_5)/5$ in $\mathrm{H}^1(\bbQ_5, E^a_{5})$ is a maximal isotropic subspace for the pairing~\eqref{eq:cup}.
\end{proof}

 To prove that $\varphi_*(\chi_Q)$ is not a scalar multiple of $\chi_P$ (with the notation of Proposition~\ref{prop:suffice5}), we wish to apply the following special case of a result of Harpaz and Skorobogatov.
 
 \begin{proposition}[{\cite[Corollary~3.7]{HS}}]\label{prop:HS}
 Let $k$ be a field of characteristic zero 
 and let $M$ be a finite simple $\Gamma_k$-module, identified with the group of $\bar{k}$-points of
a finite \'{e}tale commutative group $k$-scheme $\cG_M$. 
 Let $K$ be the smallest extension of $k$ such that $\Gamma_K$ acts trivially on $M$, let $G=\Gal(K/k)$ and suppose that $\mathrm{H}^1(G,M)=0$. 
 Let $\alpha,\beta\in \mathrm{H}^1(k,M)$ be non-zero. Then the associated torsors $Z_\alpha$ and $Z_\beta$ for $\cG_M$ are integral $k$-schemes. Furthermore, the following conditions are equivalent:
 \begin{enumerate}
\item there exists $r \in R:=\End_G(M)$ such that $r\alpha = \beta$;
\item $R\alpha = R\beta \subset \mathrm{H}^1(k, M)$;
\item $Z_\alpha$ and $Z_\beta$ are isomorphic as abstract $k$-schemes.
\end{enumerate}
 \end{proposition}

 Lemmas~\ref{lem:Gsame} and~\ref{lem:HSconditions} below are used to show that the hypotheses relevant for our application of Proposition~\ref{prop:HS} are satisfied. 
 
 \begin{lemma}\label{lem:Gsame}
Let $a\in\bbQ^\times$ and let $G$ and $H$ be the images of $\Gamma_{\bbQ}$ and $\Gamma_{\bbQ_5}$ in $\Aut(E^{a}_5)$, respectively. Then $[G:H]$ divides $3$.
\end{lemma}

\begin{proof}
We adapt the strategy of the proof of~\cite[Lemma~2.1]{IScorrigendum}.
Multiplying by a $6$th power, we may assume that $a\in\bbZ\setminus\{0\}$.
 Let $L=\bbQ(E^{a}_5)$, so $G=\Gal(L/\bbQ)$. 
For any $E$-primary prime $\pi\in\bbZ[\zeta_3]$ that is coprime to $2\cdot 3\cdot 5\cdot a$ and unramified in $L/\bbQ(\zeta_3)$, the action of $\Frob_\pi\in\Gamma_{\bbQ(\zeta_3)}$ on $E^{a}_5$ is given by multiplication by (the reduction modulo $5$ of) 
\begin{equation}\label{eq:Frobaction}
\pm\overline{\left(\frac{4a}{\pi}\right)}_6\pi,
\end{equation}
see~\cite[Example II.10.6]{SilvermanAdv}. Thus, the action of $\Gamma_{\bbQ(\zeta_3)}$ on $E^{a}_5$ factors through a homomorphism to $\bbF_5[\zeta_3]^\times=\bbF^\times_{25}$. 

Let $F\subset L$ be the fixed field of the kernel of the action of $\Gamma_\bbQ$ on $\End(E^{a}_5)$. The natural map $\GL(E^{a}_5)\to \PGL(E^{a}_5)$ restricts to a surjective homomorphism $\Phi: G\to\Gal(F/\bbQ)$ with kernel $\bbF_5^\times$.
By Lemma~\ref{lem:+-}, $\End(E^{a}_5)=(\bbZ[\zeta_3]/5)\oplus\End(E^{a}_5)^-$. Therefore, $\bbQ(\zeta_3)\subset F$.

The proof of Proposition~\ref{prop:gen} shows that $\Frob_\pi\in \Gal(L/\bbQ(\zeta_3))$ acts on $\End(E^{a}_5)^-$ as left-multiplication by $\left(\frac{2^4\cdot 5\cdot a^2}{\pi}\right)_6^{-1}$. By the definition of the sextic residue symbol and the fact that these Frobenius elements generate $\Gal(L/\bbQ(\zeta_3))$, this means that the action of $\Gal(L/\bbQ(\zeta_3))$ on $\End(E^{a}_5)^-$ is given by the sextic character attached to $(2^4\cdot 5\cdot a^2)^{-1}$, which sends $\sigma\in \Gamma_{\bbQ(\zeta_3)}$ to $\frac{\sqrt[6]{2^4\cdot 5\cdot a^2}}{\sigma(\sqrt[6]{2^4\cdot 5\cdot a^2})}$. Therefore, $F=\bbQ(\zeta_3,\sqrt[6]{2^4\cdot 5\cdot a^2})$.

Now fix an inclusion $\Gamma_{\bbQ_5}\subset \Gamma_{\bbQ}$. Let $K=\bbQ_5(E^{a}_5)$, so $H=\Gal(K/\bbQ_5)$. Let $F_5=\bbQ_5(\zeta_3,\sqrt[6]{2^4\cdot 5\cdot a^2})$ be the fixed field of the kernel of the action of $\Gamma_{\bbQ_5}$ on $\End(E^{a}_5)$. Write $a=5^ru$ for $u\in\bbZ_5^\times$ and use that every unit in $\bbZ_5^\times$ is a cube to see that $F_5=\bbQ_5(\zeta_3,\sqrt[6]{5^{2r+1}})$. Therefore, $[F_5:\bbQ_5(\zeta_3)]$ is either $2$ or $6$ depending on whether $r\equiv 1\bmod{3}$. Thus, $\Phi(H)=\Gal(F_5/\bbQ_5)$ has index dividing $3$ in $\Phi(G)=\Gal(F/\bbQ)$. To complete the proof, it remains to show that $\bbF_5^\times=\ker(\Phi)\subset H$.

Consider the restriction of $\Phi$ to $G':=\Gal(L/\bbQ(\zeta_3))$. This is a cyclic group of order dividing $|\bbF_{25}^\times|=24$, which contains $\bbF_5^\times$ as its unique subgroup of order $4$, and $\Phi:G'\to \Gal(F/\bbQ(\zeta_3))$ is the quotient by $\bbF_5^\times$. It suffices to show that $\bbF_5^\times\subset H':= \Gal(K/\bbQ_5(\zeta_3))$. We have 
\[H'/(H'\cap \bbF_5^\times) \cong \Phi(H')=\Gal(F_5/\bbQ_5(\zeta_3)).\]
Therefore, $[F_5:\bbQ_5(\zeta_3)]$ divides $|H'|$ (whereby $H'$ has even order) and $H'\not\subset \bbF_5^\times$. Since $H'$ is a subgroup of the cyclic group $G'$ of order dividing $24$, we can now conclude that $\bbF_5^\times\subset H'$ unless $H'$ has order $6$. Suppose for contradiction that $|H'|=6$. Then $|H'\cap \bbF_5^\times|=2$ and $[F_5:\bbQ_5(\zeta_3)]=|H'/(H'\cap \bbF_5^\times)|=3$, a contradiction.
\end{proof}

\begin{lemma}\label{lem:HSconditions}
Let $a\in\bbQ_5^\times$ 
and let $H$ denote the image of $\Gamma_{\bbQ_5}$ in $\Aut(E^{a}_5)$. 
Then 
\begin{enumerate}
\item $\mathrm{H}^1(H, E^{a}_5)=0$; \label{item:H1triv}
\item $E^{a}_5$ is a simple $\Gamma_{\bbQ_5}$-module;\label{item:simple}
\item $\End_H(E^{a}_5)=\bbF_{5}$. \label{item:EndF5}
\end{enumerate}
\end{lemma}

\begin{proof}
Multiplying by a $6$th power in $\bbQ_5^\times$, we may assume that $a\in\bbZ\setminus\{0\}$. Let $G$ denote the image of $\Gamma_{\bbQ}$ in $\Aut(E^{a}_5)$. Recall from the proof of Lemma~\ref{lem:Gsame} that the action of $\Gamma_{\bbQ(\zeta_3)}$ on $E^{a}_5$ factors through a homomorphism to $\bbF_5[\zeta_3]^\times=\bbF^\times_{25}$. Therefore, $G$ has order dividing $2\cdot 24$. Now~\eqref{item:H1triv} is immediate since $H\subset G$ is a finite group of order coprime to $5$ and $\mathrm{H}^1(H, E^{a}_5)$ is killed by $|H|$ and by $5$.
 
The strategy of the rest of the proof is to use Lemma~\ref{lem:Gsame} to move from the action of $\Gamma_{\bbQ_5}$ on $E^{a}_5$ to that of $\Gamma_\bbQ$, and from there to the action of $\Gamma_{\bbQ(\zeta_3)}$, where we have the explicit action of Frobenius elements on $E^{a}_5$ given by~\eqref{eq:Frobaction}.
Henceforth, let $\pi\in\bbZ[\zeta_3]$ be an E-primary prime lying above a rational prime $p$ that is coprime to $2\cdot 3\cdot 5\cdot a$ and completely split in $\bbQ(\zeta_3, \sqrt[3]{4a})/\bbQ$. Then~\eqref{eq:Frobaction} shows that the action of $\Frob_\pi$ on $E^{a}_5$ is given by multiplication by the reduction modulo $5$ of $\pm\pi=x+y\zeta_3$ for some $x,y\in\bbZ$. Assume in addition that $p$ is inert in $\bbQ(\sqrt{5})/\bbQ$. Then, writing $p=\pi\bar{\pi}=x^2-xy+y^2$, we see that if $y\equiv 0\pmod{5}$ then $p\equiv x^2\mod{5}$ and hence quadratic reciprocity gives $\left(\frac{5}{p}\right)=\left(\frac{p}{5}\right)=1$, contradicting the fact that $p$ is inert in $\bbQ(\sqrt{5})/\bbQ$. Therefore, $y\not\equiv 0\pmod{5}.$ Similarly, $x(x-y)\not\equiv 0\pmod{5}.$ By Lemma~\ref{lem:Gsame}, the image of $\Frob_{\pi}^3$ in $\Aut(E^{a}_5)$ is contained in $H$. It acts on $E^{a}_5$ as multiplication by the reduction modulo $5$ of $(x+y\zeta_3)^3=x^3+y^3-3xy^2+3xy(x-y)\zeta_3$.

Let $M$ be a $\Gamma_{\bbQ_5}$-submodule of $E^{a}_5$ and suppose $T\in M$ is non-zero. We have $\Frob_\pi^3\cdot T=(x^3+y^3-3xy^2+3xy(x-y)\zeta_3)\cdot T\in M$. Since $3xy(x-y)\not\equiv 0\pmod{5}$, 
$T$ and $(x^3+y^3-3xy^2+3xy(x-y)\zeta_3)\cdot T$ form an $\bbF_5$-basis of $E^{a}_5$ and hence $M=E^{a}_5$, proving~\eqref{item:simple}.
 
For \eqref{item:EndF5}, Propositions~\ref{prop:SZhom} and~\ref{prop:gen} show that 
\[\End_{\Gamma_{\bbQ}}(E^{a}_5)=(\End(\bar{E}^a)/5)^{\Gamma_\bbQ}=(\bbZ[\zeta_3]/5)^{\Gamma_\bbQ}=\bbF_5.\]
Thus, it suffices to show that $\End_{H}(E^{a}_5)=\End_{\Gamma_{\bbQ}}(E^{a}_5)$. Clearly, 
\[\End_{\Gamma_{\bbQ}}(E^{a}_5)=\End_{G}(E^{a}_5)\subset \End_{H}(E^{a}_5)\]
so it suffices to show the reverse inclusion. Let $\varphi\in \End_{H}(E^{a}_5)$. Then for $\pi$ as above, $\varphi$ commutes with the image of $\Frob_{\pi}^3$ in $\Aut(E^{a}_5)$, which implies that $\varphi$ commutes with the image of $\zeta_3$ in $\Aut(E^{a}_5)$. Therefore, $\varphi$ commutes with the image of $\Gamma_{\bbQ(\zeta_3)}$ in $\Aut(E^{a}_5)$. Let $\tau\in \Gamma_\bbQ$ denote complex conjugation, which generates $\Gamma_\bbQ/\Gamma_{\bbQ(\zeta_3)}$. Since $\tau$ has order $2$, Lemma~\ref{lem:Gsame} shows that the image of $\tau$ in $\Aut(E^{a}_5)$ lies in $H$. Therefore, $\varphi$ commutes with $\tau$ and hence with the whole image of $\Gamma_\bbQ$ in $\Aut(E^{a}_5)$, as required.
\end{proof}

For $\chi\in \mathrm{H}^1(\bbQ_5, E^c_{5})$, write $K(\chi)$ for the function field of the $\bbQ_5$-torsor for the group $\bbQ_5$-scheme $E^c_5$ corresponding to $\chi$. Note that $K(\chi)$ is a finite extension of $\bbQ_5$.

\begin{proposition}\label{prop:disc}
Let $c,d,\chi_P,\chi_Q$ be as in Proposition~\ref{prop:suffice5}.
To prove Theorem~\ref{thm:mainWA} for $Y=\Kum(E^c\times E^d)$ and $\ell=5$, it suffices to show that $K(\chi_{P})$ is not isomorphic to $K(\chi_Q)$. In particular, it suffices to show that $\disc(K(\chi_{P}))$ is not equal to $\disc(K(\chi_{Q}))$.
\end{proposition}

\begin{proof}
Proposition~\ref{prop:suffice5} tells us that it suffices to show that $\varphi_*(\chi_Q)$ is not a scalar multiple of $\chi_P$. By Lemma~\ref{lem:HSconditions}, we can apply Proposition~\ref{prop:HS} to see that this is equivalent to showing that the $\bbQ_5$-torsor for the group $\bbQ_5$-scheme $E^c_5$ determined by $\varphi_*(\chi_Q)$ is not isomorphic (as an abstract $\bbQ_5$-scheme) to that determined by $\chi_P$. For this, it suffices to show that the function fields of the torsors are not isomorphic as extensions of $\bbQ_5$. Since $\varphi$ is an isomorphism (by Lemma~\ref{lem:varphigen}), it follows from the construction of the pushforward that $\chi_Q$ and $\phi_*(\chi_Q)$ are isomorphic as $\bbQ_5$-schemes.
\end{proof}

Note that if we prove Theorem~\ref{thm:mainWA} with $\ell=5$ and $Y=\Kum(E^c\times E^d)$ for a given pair of rational numbers $c$ and $d$, then it also holds for all multiples of $c$ and $d$ by elements in $\bbQ^\times\cap \bbQ_5^{\times 6}$, since the relevant elliptic curves are isomorphic over $\bbQ_5$. Hence, we can reduce to considering $c,d\in\{2^i\cdot5^j\mid 0\leq i\leq 1, 0\leq j\leq 5\}$, since these elements represent all cosets in $\bbQ_5^\times/\bbQ_5^{\times 6}$. For $a\in \{2^i\cdot5^j\mid 0\leq i\leq 1, 0\leq j\leq 5\}$, Table~\ref{table:cd} records a generator $P_a$ for $E^a(\bbQ_5)/5$ and the discriminant of the function field of the torsor $\chi_{P_a}$, which we denote by $\disc(K(\chi_{P_a}))$. 
\begin{table}[h!]
\centering
\begin{tabular}{|c|c|c|}
 \hline
 $a$ & $P_a$  & $\disc(K(\chi_{P_a}))$ \\
 \hline
 $1$ & $(\frac{1}{5^2},\frac{\sqrt{1+5^6}}{5^3})$  & $5^{25}$ \\
  $5$ & $(1,\sqrt{1+5})$  & $5^{45}$  \\
  $5^2$ & $(1,\sqrt{1+5^2})$  & $5^{41}$ \\
  $5^3$ & $(1,\sqrt{1+5^3})$ & $5^{37}$ \\
  $5^4$ & $(1,\sqrt{1+5^4})$ &  $5^{33}$\\
  $5^5$ & $(1,\sqrt{1+5^5})$ & $5^{25}$\\
   $2$ & $(\frac{1}{5^2},\frac{\sqrt{1+2\cdot 5^6}}{5^3})$  & $5^{25}$ \\
  $2\cdot 5$ & $(1,\sqrt{1+2\cdot 5})$  & $5^{45}$  \\
  $2\cdot 5^2$ & $(1,\sqrt{1+2\cdot 5^2})$  & $5^{41}$ \\
  $2\cdot 5^3$ & $(1,\sqrt{1+2\cdot 5^3})$ & $5^{37}$ \\
  $2\cdot 5^4$ & $(1,\sqrt{1+2\cdot 5^4})$ &  $5^{33}$\\
  $2\cdot 5^5$ & $(1,\sqrt{1+2\cdot 5^5})$ & $5^{25}$\\
 \hline
\end{tabular}   
 \caption{} 
\label{table:cd}
\end{table}

Note that in all cases $\disc(K(\chi_{P_a}))=\disc(K(\chi_{P_{2a}}))$.

We are now ready to prove the last remaining case of Theorem~\ref{thm:mainWA}, in which $Y=\Kum(E^c\times E^d)$ for $c,d\in\bbQ^\times$ with $2^4\cdot 5\cdot cd\in\langle -3^3\rangle\bbQ^{\times 6}$, and $\ell=5$.

\begin{proof}[Completion of the proof of Theorem~\ref{thm:mainWA}]
Since $2^4\cdot 5\cdot cd\in\langle-3^3\rangle\cdot\bbQ^{\times 6}$, $2\in\bbQ_5^{\times 3}$ and $-3\cdot \bbQ_5^{\times 6}=2\cdot \bbQ_5^{\times 6}$, Table~\ref{table:cd} and the preceding discussion show that $\disc(K(\chi_{P_d}))=\disc(K(\chi_{P_{5^5c^5}}))$, where we can take $5^5c^5$ modulo $\bbQ_5^{\times 6}$. Let $S = \bbQ_5^{\times 6}\cup 2\cdot \bbQ_5^{\times 6}\cup 5^5\cdot \bbQ_5^{\times 6}\cup 2\cdot 5^5\cdot\bbQ_5^{\times 6}$. 
An inspection of Table~\ref{table:cd} shows that $\disc(K(\chi_{P_c}))$ is not equal to $\disc(K(\chi_{P_{5^5c^5}}))$ unless $c\in S$. Thus, by Proposition~\ref{prop:disc}, we have proved Theorem~\ref{thm:mainWA} for $c\notin S$. 
Now suppose that $c\in S$. Then $5^3c\notin S$ so Theorem~\ref{thm:mainWA} holds for $\Kum(E^{5^3c}\times E^{5^3d})$. Now recall that $\Kum(E^{c}\times E^{d})\cong \Kum(E^{5^3c}\times E^{5^3d})$, so the proof of Theorem~\ref{thm:mainWA} is complete.
\end{proof}

\end{document}